\documentclass[12pt]{article}
\usepackage{amsmath,amsthm,amssymb}
\usepackage[hang]{caption2}
\usepackage{indentfirst}
\usepackage[section]{placeins}
\numberwithin{equation}{section}     
\usepackage{graphics}
\usepackage{pdfsync}
\usepackage{hyperref}

\def\accentsfrancais{applemac}
  \RequirePackage[\accentsfrancais]{inputenc}

\newtheorem{thm}{Theorem}
\newtheorem{theorem}[thm]{Theorem}
\numberwithin{thm}{section}
\newtheorem{lemma}[thm]{Lemma}
\newtheorem{cor}[thm]{Corollary}

\newtheorem{remark}[thm]{Remark}

\def\eps{\varepsilon}
\def\ds{\displaystyle}
\def\emptyset{/\kern-.51em o}
\def\eq{\mathop{\vrule height2,6pt depth-2,3pt
         width -1pt\kern 0pt =}}

\let\norbali\normalbaselines
\def\anorbali{\norbali\advance\lineskip\jot
\advance\baselineskip\jot\advance\lineskiplimit\jot}
\def\ouvre{\let\normalbaselines\anorbali}
\def\D{{\mathbb{D}}}

\def\D{{\mathbb{D}}}
\def\R{{\mathbb{R}}}

\def\N{\rm \hbox{I\kern-.2em\hbox{N}}}
\def\Z{\rm \hbox{Z\kern-.3em\hbox{Z}}}

\def\Gd{{\bf d}}
\def\Ge{{\bf e}}

\def\Gr{{\bf r}}

\def\Gx{{\bf x}}

\def\GA{{\bf A}}

\def\GE{{\bf E}}

\def\GG{{\bf G}}

\def\GI{{\bf I}}

\def\GQ{{\bf Q}}
\def\GR{{\bf R}}

\def\GX{{\bf X}}

\begin{document}
\title{Junction between  a  plate and a rod of comparable thickness in nonlinear elasticity. Part II}
\author{D. Blanchard$^{1}$, G. Griso$^{2}$}
\date{}
\maketitle
 
{\footnotesize
\begin{center}

$^{1}$  Universit\'e de Rouen,  France. DECEASED on February 11th, 2012\footnote{We were just finishing this paper when suddenly two days later my friend the Professor Dominique Blanchard died. We worked seven years together, our collaboration was very successful for both.  }.

$^{2}$ Laboratoire J.-L. Lions--CNRS, Bo\^\i te courrier 187, Universit\'e  Pierre et
Marie Curie,\\ 4~place Jussieu, 75005 Paris, France, \; Email: griso@ann.jussieu.fr\\

\end{center} }

\begin{abstract}
 We analyze the asymptotic behavior of a junction problem between a plate and a perpendicular rod made  of a nonlinear elastic material. The two parts of this multi-structure have small thicknesses of the same order $\delta$.  We use the  decomposition techniques obtained for the large deformations and the displacements in order to derive the limit energy as $\delta$ tends to $0$. 
  
\end{abstract}
 


\section{ Introduction} 
\medskip
In  a former paper \cite{DB_GGI} we derive the limit energy of the junction problem between a plate and a rod under an assumption that couples  their respective thicknesses $\delta$ and $\eps$ to the order of the Lam\'e's coefficients of the materials in the plate and  in the rod. This assumption precludes the case where the thicknesses have the same order and the structure is made of the same material (see equation 1.1 in the introduction of \cite{DB_GGI}). The aim of the present paper is to analyze this specific case for a total energy of order $\delta^5$. As in \cite{DB_GGI}, the structure is clamped on a part of the lateral boundary of the plate and it is free on the rest of its boundary.

The main difference here is the behavior in the rod in which, for this level of energy (which is higher than the  maximum allowed in \cite{DB_GGI}), the stretching-compression is of order $\delta$ while the bending is of order $\delta^{1/2}$. The most important  consequence is that in the limit model for the rod  the  stretching-compression is actually given by the bending  in the rod (through a nonlinear relation) and by the bending  in the plate at the junction point (see \eqref{VIs} and \eqref{V=WQ(0)}). The bending and torsion models   in the rod are  the standard linear ones. In the plate the limit model is the  Von K\'arm\'an system in which the action of the rod is  modelized by a   punctual force at the junction.
 
Let us emphasize that in order to obtain sharp estimates on the deformations in the junction area, see Lemma \ref{lemme2}, we use the decomposition techniques in thin domains (see \cite{DB_GGI},\cite{GDecomp}, \cite{BGJE}, \cite{BGRod}).   In order to scale the applied forces which induce a total energy energy of order $\delta^5$, from Lemma \ref{lemme2} and \cite{BGRod}, we derive a nonlinear Korn's inequality for the rod (as far as the plate is concerned this type of inequality is already established in \cite{DB_GGI}). The nonlinear character of these Korn's inequalities  prompt us to adopt  smallness assumptions on some components of the forces. Then, we are in a position to study the asymptotic behavior of the Green-St Venant's strain tensors 
in the two parts of the structure. At last this allows us to characterize the limit of the rescaled infimum of the 3d energy as the minimum of a functional over a set of limit admissible displacements which includes the nonlinear relation between the stretching-compression and the bending  in the rod.

In Section 2 we introduce a few general notations. Section 3 gives a few recalls on  the decomposition technique of the deformations in thin structures. In Section 4, we derive  first estimates on the terms of the decomposition of a deformation in the rod and sharp estimates in  the junction area. In the same section we also obtain Korn's inequality in the rod.  In Section 5 we introduce the elastic energy and the assumptions on the applied forces in order to obtain a total elastic energy of order $\delta^5$. In Section 6 we analyze the asymptotic behavior of the Green-St-Venant's strain tensors in the plate and in the rod. In Section 7 we prove the main result of the paper namely the characterization of the  limit of the rescaled infimum of the 3d energy.

As general references on the theory of elasticity we refer to \cite{Ant1} and \cite{C1}. The reader is referred to \cite{Ant}, \cite{Trab}, \cite{GROD} for an introduction of rods models and to \cite{C11}, \cite{C2}, \cite{Ciarlet3}, \cite{FJM}, \cite{MS} for plate models. As far as junction problems in multi-structures we refer to \cite{CDN}, \cite{C2}, \cite{DJR}, \cite{DJP}, \cite{Ldret},  \cite{Auf}, \cite{Gru1}, \cite{Gru2}, \cite{GSP}, \cite{Murat}, \cite{Gaud}, \cite{BGG1}, \cite{BGG2}, \cite{BG1}, \cite{GSR}, \cite{BGNOTE}, \cite{BGELAS}. For the decomposition method in thin structures we refer to \cite{GROD}, \cite{GSP}, \cite{GDecomp}, \cite{GSR}, \cite{BGRod}, \cite{BGJE}, \cite{SimplCoq}, \cite{BGELAS}.

\vskip 1mm
\section { Notations.}
\vskip 1mm
Let us introduce a few  notations and definitions concerning the geometry of the plate and the rod. Let $\omega$ be a bounded domain in $\R^2$ with lipschitzian boundary included in the plane $(O; \Ge_1, \Ge_2)$ and such that $O\in \omega$. 
The plate is the domain
 $$ \Omega_\delta=\omega\times ]-\delta,\delta[.$$
Let $\gamma_0$ be an open subset of $\partial \omega$ which is made of a finite number of connected components (whose closure are disjoint). The corresponding lateral part of the boundary of $\Omega_\delta$ is
$$\Gamma_{0,\delta}=\gamma_0\times]-\delta,\delta[.$$ The rod is defined by  
$$B_{\delta}=D_\delta \times ]-\delta,L[,\qquad D_\delta=D(O,\delta),\qquad D=D(O,1)$$ where $\delta>0$ and where $D_r=D(O,r)$ is the disc of radius $r$ and center the origin $O$. We assume that $D\subset\subset \omega$. The whole structure is denoted
 $${\cal S}_{\delta}=\Omega_\delta\cup B_{\delta}$$
 while the junction is 
 $$C_{\delta}=\Omega_\delta\cap B_{\delta}=D_\delta \times ]-\delta,\delta[.$$
We denote  $I_d$  the identity map of $\R^3$.   The set of admissible deformations of the structure is
$$\D_\delta=\Big\{ v\in H^1({\cal S}_{\delta} ; \R^3)\;\;|\;\; v=I_d\enskip\hbox{on}\enskip\Gamma_{0,\delta}\Big\}.$$
The Euclidian norm in $\R^k$ ($k\ge 1$) will be denoted $|\cdot|$ and the Frobenius norm of a square matrix will be denoted $|||\cdot|||$.
\section { Some recalls.}
\vskip 1mm
\noindent To any vector $F\in \R^3$ we associate the antisymmetric matrix $\GA_F$ defined by
\begin{equation}\label{1}
\forall \Gx\in \R^3,\qquad \GA_F\, \Gx=F\land \Gx.
\end{equation}
From now on, in order to simplify the notations, for any open set ${\cal O}\subset \R^3$ and any field $u\in H^1({\cal O} ; \R^3)$, we set
$$\GG_s(u,{\cal O})=||\nabla u+(\nabla u)^T||_{L^2({\cal O};\R^{3\times 3})}$$ and
$$\Gd(u,{\cal O})=||dist(\nabla u,SO(3))||_{L^2({\cal O})}.$$
\subsection{Recalls on the  decompositions of the plate-displacement.}

\noindent  We know (see \cite{GSP} or \cite{GDecomp}) that any displacement $u\in   H^1(\Omega_\delta; \R^3)$ of the plate is decomposed as 
\begin{equation}\label{FDec}
u(x)={\cal U} (x_1,x_2)+ x_3{\cal R}(x_1,x_2) \land \Ge_3+\overline{u} (x),\qquad x\in \Omega_\delta
\end{equation}
where ${\cal U}$ is defined by
$${\cal U}(x_1,x_2)={1\over 2\delta}\int_{-\delta}^\delta u(x_1,x_2,x_3)dx_3\qquad \hbox{for a.e. } \; x_3\in \omega$$ and where ${\cal R}$ is also defined via an average involving the displacement  $u$ (see \cite{GSP} or \cite{GDecomp}). The fields ${\cal U}$ and ${\cal R}$ belong to $ H^1(\omega; \R^3)$ and $\overline{u} $ belongs to $  H^1(\Omega_\delta; \R^3)$.  The sum of the two first terms $U_e(x)={\cal U} (x_1,x_2)+x_3{\cal R}(x_1,x_2)\land \Ge_3$ is called the elementary displacement associated to $u$.

\noindent  The following Theorem is proved in \cite{GSP} for the displacements in $H^1(\Omega_\delta; \R^3)$ and in \cite{GDecomp} for the displacements in $W^{1,p}(\Omega_\delta; \R^3)$ ($1<p<+\infty$).
\begin{theorem}\label{Theorem 3.3.}
Let $u\in   H^1(\Omega_\delta; \R^3)$, there exists  an elementary displacement $U_e(x)={\cal U}(x_1,x_2)+x_3{\cal R}(x_1,x_2)\land \Ge_3$ and a warping  $\overline{u}$  satisfying \eqref{FDec} such that 
\begin{equation}\label{3.7}
\begin{aligned}
&||\overline{u} ||_{L^2(\Omega_\delta; \R^3)}\le C\delta \GG_s(u,\Omega_\delta),\quad||\nabla \overline{u} ||_{ L^2(\Omega_\delta; \R^3)}\le C \GG_s(u,\Omega_\delta),\\
&\Bigl\|{\partial {\cal R}\over \partial x_\alpha}\Big\|_{ L^2(\omega; \R^3)}\le {C\over \delta^{3/2}} \GG_s(u,\Omega_\delta),\\
& \Bigl\|{\partial{\cal U}\over \partial x_\alpha}-{\cal R}\land \Ge_\alpha\Big\|_{ L^2(\omega; \R^3)}\le {C\over \delta^{1/2}}\GG_s(u,\Omega_\delta),\\
& ||\nabla u-\GA_{\cal R}||_{ L^2(\omega; \R^9)}\le C\GG_s(u,\Omega_\delta),
\end{aligned}
\end{equation} where the constant $C$ does not depend on $\delta$.
\end{theorem}
 The warping $\overline{u}$ satisfies the following relations
 \begin{equation}\label{RelWarPlaque}
 \begin{aligned}
& \int_{-\delta}^\delta\overline{u}(x_1,x_2,x_3)dx_3=0,\qquad  \int_{-\delta}^\delta x_3\overline{u}_\alpha(x_1,x_2,x_3)dx_3=0\quad \hbox{for a.e. } (x_1,x_2)\in \omega.
\end{aligned} \end{equation}
If a deformation  $v$ belongs to $\D_\delta$ then  the displacement $u=v-I_d$ is equal to $0$ on  $\Gamma_{0,\delta}$. In this case  the the fields ${\cal U}$, ${\cal R}$  and the warping $\overline{u}$ satisfy
\begin{equation}\label{CLUR}
{\cal U}= {\cal R}=0\qquad \hbox{on } \enskip \gamma_0,\qquad \overline{u}=0\qquad \hbox{on}\quad \Gamma_{0,\delta}.
\end{equation}
 Then,  from   \eqref{3.7}, for any deformation $v\in \D_\delta$   the corresponding displacement $u=v-I_d$ verifies the following estimates (see also \cite{GSP}):
\begin{equation}\label{Estm}
\begin{aligned}
||{\cal R}||_{H^1(\omega; \R^3)}+||{\cal U}_3||_{H^1(\omega)}\le {C\over \delta^{3/2}} \GG_s(u,\Omega_\delta),\\
 ||{\cal R}_3||_{L^2(\omega)}+||{\cal U}_\alpha||_{H^1(\omega)}\le {C\over \delta^{1/2}} \GG_s(u,\Omega_\delta).
\end{aligned} \end{equation} The constants depend only on $\omega$. From the above estimates we deduce the following Korn's type inequalities for the displacement $u$
\begin{equation}\label{KoP0}
\begin{aligned}
&||u_\alpha||_{L^2(\Omega_\delta)}\le C_0\GG_s(u,\Omega_\delta),\qquad ||u_3||_{L^2(\Omega_\delta)}\le {C_0\over \delta}\GG_s(u,\Omega_\delta),\\
& ||u-{\cal U}||_{L^2(\Omega_\delta ; \R^3)}\le {C\over \delta}\GG_s(u,\Omega_\delta),\quad ||\nabla u||_{L^2(\Omega_\delta;\R^9)}\le {C\over \delta}\GG_s(u,\Omega_\delta).
 \end{aligned}
\end{equation}
 Due to Theorem 3.3  established in \cite{BGJE}, the displacement  $u=v-I_d$ is also decomposed as 
\begin{equation}\label{FDec8}
u(x)={\cal U} (x_1,x_2)+ x_3(\GR(x_1,x_2)-\GI_3) \Ge_3+\overline{\overline{u}} (x),\qquad x\in \Omega_\delta
\end{equation} where  $\GR\in H^1(\omega; \R^{3\times 3})$, $\overline{\overline{u}} \in  H^1(\Omega_\delta; \R^3)$ and we have the following estimates
\begin{equation}\label{87}
\begin{aligned}
&||\overline{\overline{u}} ||_{L^2(\Omega_\delta; \R^3)}\le C\delta \Gd(v,\Omega_\delta)\qquad ||\nabla \overline{\overline{u}} ||_{ L^2(\Omega_\delta; \R^9)}\le C \Gd(v,\Omega_\delta)\\
&\Bigl\|{\partial \GR\over \partial x_\alpha}\Big\|_{ L^2(\omega; \R^9)}\le {C\over \delta^{3/2}} \Gd(v,\Omega_\delta)\\
& \Bigl\|{\partial{\cal U}\over \partial x_\alpha}-(\GR-\GI_3) \Ge_\alpha\Big\|_{ L^2(\omega; \R^3)}\le {C\over \delta^{1/2}}\Gd(v,\Omega_\delta)\\
& \bigl\|\nabla  v-\GR \big\|_{ L^2(\Omega_\delta; \R^9)}\le C\Gd(v,\Omega_\delta)
\end{aligned}
\end{equation} where the constant $C$ does not depend on $\delta$. The following boundary conditions are satisfied
\begin{equation}\label{CLUR8}
{\cal U}=0,\quad  \GR=\GI_3\qquad \hbox{on } \enskip \gamma_0,\qquad \overline{\overline{u}}=0\qquad \hbox{on}\quad \Gamma_{0,\delta}.
\end{equation} Due to \eqref{87} and the above boundary conditions we obtain
\begin{equation}\label{88}
 ||\GR-\GI_3||_{H^1(\omega; \R^9)}+ ||{\cal U}||_{H^1(\omega; \R^3)}\le {C\over \delta^{3/2}}\Gd(v,\Omega_\delta).
\end{equation} 
\subsection{Recall on the  decomposition of the rod-deformation.}

Now, we consider a deformation  $v\in H^1(B_\delta ; \R^3)$ of the rod $B_\delta$. This deformation can be   decomposed as (see Theorem 2.2.2 of \cite{BGRod})
\begin{equation}\label{DecR}
v(x)={\cal V}(x_3)+\GQ(x_3)\big(x_1\Ge_1+x_2\Ge_2\big)+\overline{\overline{v}}(x),\qquad x\in B_\delta,
 \end{equation}
where $\ds{\cal V}(x_3) ={1\over |D_\delta|}\int_{D_\delta}v(x)dx_1dx_2$ belongs to $H^1(-\delta,L;\R^3)$, where  $\GQ$  belongs to $ H^1(-\delta,L;SO(3))$ and $\overline{\overline{u}}$ belongs to $ H^1(B_\delta;\R^3)$. Let us give a few comments on the above decomposition. The term ${\cal V}$ gives the deformation of the center line of the rod. The second term $\GQ(x_3)\big(x_1\Ge_1+x_2\Ge_2\big)$ describes the rotation of the cross section (of the rod) which contains the point $(0,0,x_3)$. The sum of the terms ${\cal V} (x_3)+\GQ(x_3)\big(x_1\Ge_1+x_2\Ge_2\big)$ is called an elementary deformation of the rod.

The following theorem (see Theorem 2.2.2 of \cite{BGRod}) gives a decomposition \eqref{DecR} of a deformation and estimates on the terms of this decomposition.
\smallskip
\begin{theorem}\label{Theorem II.2.2.}  Let $v\in H^1(B_\delta;\R^3)$, there exists an elementary deformation ${\cal V}(x_3) +\GQ(x_3) \big(x_1\Ge_1 +x_2\Ge_2 \big)$ and a warping $\overline{\overline{v}}$ satisfying \eqref{DecR} and such that 
\begin{equation}\label{EstmRod}
\begin{aligned}
&||\overline{\overline{v}}||_{L^2(B_\delta;\R^3)}\le C\delta\Gd(v, B_\delta),\\
&||\nabla\overline{\overline{v}}||_{L^2(B_\delta;\R^{3\times 3})}\le C \Gd(v, B_\delta),\\
&\Bigl\|{d\GQ\over dx_3}\Big\|_{L^2(-\delta, L ;\R^{3\times 3})}\le {C\over\delta^2} \Gd(v, B_\delta),\\
& \Bigl\|{d{\cal V}\over dx_3}-\GQ\Ge_3\Big\|_{L^2(-\delta,L;\R^3)}\le {C\over \delta}\Gd(v, B_\delta),\\
& \bigl\|\nabla v-\GQ \big\|_{L^2(B_\delta;\R^{3\times 3})}\le C\Gd(v, B_\delta),
\end{aligned}
\end{equation}
 where the constant $C$ does not depend on $\delta$ and $L$.
 \end{theorem}

\section{Preliminaries results}
Let $v$ be a deformation in $\D_\delta$. We set  $u=v-I_d$. We decompose $u$ as \eqref{FDec} and \eqref{FDec8} in the plate and we decompose the deformation $v$ as  \eqref{DecR} in the rod.
\subsection{A complement to the mid-surface bending.}
\noindent Let us set 
$$H^1_{\gamma_0}(\omega)=\{\varphi\in H^1(\omega)\; ;\; \varphi=0\; \hbox  { on } \gamma_0\}.$$
\vskip 1mm
\noindent We define the  function $\widetilde{\cal U}_3$ as the solution of the following variational problem :
\begin{equation}\label{defVtilde2}
\left\{\begin{aligned}
& \widetilde{\cal U}_3\in H^1_{\gamma_0}(\omega),\\
&\int_\omega\nabla\widetilde{\cal U}_3\nabla\varphi=\int_\omega(\GR-\GI_3)\Ge_\alpha\cdot\Ge_3 {\partial\varphi\over \partial x_\alpha},\\
&\forall\varphi\in H^1_{\gamma_0}(\omega)
\end{aligned}\right.
\end{equation} where $\GR$ appears in the decomposition \eqref{FDec8} of $u$.
Due to \eqref{87}-\eqref{88}, the function  $\widetilde{\cal U}_3$  belongs to $H^1_{\gamma_0}(\omega)\cap H^2(D)$ (remind that $D$ is the disc of radius 1 and center the origin $O$; and we assumed that $D\subset\subset \omega$). The function  $\widetilde{\cal U}_3$ satisfies  the  estimates:
\begin{equation}\label{4002}
\begin{aligned}
&||\widetilde{\cal U}_3||_{H^1(\omega)}\le {C\over \delta^{3/2}}\Gd(v,\Omega_{\delta}),\quad||{\cal U}_3-\widetilde{\cal U}_3||_{H^1(\omega)}\le {C\over \delta^{1/2}}\Gd(v,\Omega_{\delta}),\\
&||\widetilde{\cal U}_3||_{H^2(D)}\le {C\over \delta^{3/2}}\Gd(v,\Omega_{\delta}),\quad \Big\|{\partial\widetilde{\cal U}_3\over \partial x_\alpha}-(\GR-\GI_3)\Ge_\alpha\cdot\Ge_3\Big\|_{H^1(D)}\le {C\over \delta^{3/2}}\Gd(v,\Omega_{\delta}),\\
&|\widetilde {\cal U}_3(0,0)| \le   {C\over \delta^{3/2}}\Gd(v,\Omega_{\delta}).
\end{aligned}
\end{equation} The constants do not depend on $\delta$.

\subsection{A complement to the rod center-line displacement.}
Let ${\cal V}$ given by \eqref{DecR}, we consider  ${\cal W}(x_3)={\cal V}(x_3)-x_3\Ge_3=\ds{1\over |D_\delta|}\int_{D_\delta}u(x)dx_1dx_2$ the rod center-line displacement. From the above Theorem \ref{Theorem II.2.2.}, the estimate below holds true
\begin{equation}\label{WQ}
\Bigl\|{d{\cal W}\over dx_3}-\big(\GQ-\GI_3\big)\Ge_3\Big\|_{L^2(-\delta,L;\R^3)}\le {C\over \delta}\Gd(v, B_\delta).
\end{equation}
\noindent As in \cite{BGRod} we split the center line displacement ${\cal W}$ into two parts. The first one ${\cal W}^{(m)}$ stands for the main displacement of the rod which describes the displacement coming from the bending
and the second one for the stretching of the rod.
\begin{equation}\label{WBWS}
\begin{aligned}
\forall x_3\in [0,L],\qquad {\cal W}^{(m)}(x_3)&={\cal W}(0)+\int_0^{x_3}\big(\GQ(t)-\GI_3\big)\Ge_3 dt,\\
 {\cal W}^{(s)}(x_3)&={\cal W}(x_3)-{\cal W}^{(m)}(x_3).
\end{aligned}
\end{equation} 
 In the lemma below we give   estimates on  $ {\cal W}^{(s)}$ and $ {\cal W}^{(m)}$.
\begin{lemma}\label{lem33} We have
\begin{equation}\label{EstWS}
 ||{\cal W}^{(s)}||_{H^1(-\delta, L ; \R^3)}\le {C\over \delta}\Gd(v, B_\delta),
\end{equation}
 and 
\begin{equation}\label{dWB}
\big\|{\cal W}^{(m)}_{\alpha} -{\cal W}_{\alpha} (0)\big\|_{H^2(-\delta,L)} \le {C\over \delta^2} \Gd(v, B_\delta)+C|||\GQ(0)-\GI_3|||,
 \end{equation}
  \begin{equation}\label{dW3}
 \begin{aligned}
\bigl\| {\cal W}^{(m)}_{3}-{\cal W}^{(m)}_{3}(0) \Big\|_{H^1(-\delta,L)} & \le  {C\over \delta^4} \big[\Gd(v, B_\delta)\big]^2+C|\big(\GQ(0)-\GI_3\big)\Ge_3\cdot\Ge_3|,\\
\Bigl\|{d{\cal W}^{(m)}_{3}\over dx_3}\Big\|_{L^2(-\delta,\delta)} & \le  {C\over \delta^{5/2}} \big[\Gd(v, B_\delta)\big]^2+C\delta^{1/2}|\big(\GQ(0)-\GI_3\big)\Ge_3\cdot\Ge_3|.
 \end{aligned}
 \end{equation}
The constants do not depend on $\delta$.
 \end{lemma}
\begin{proof}  Taking into account the facts that ${\cal W}^{(s)}(0)=0$ and $\ds{d{\cal W}^{(s)}\over dx_3}={d{\cal W}\over dx_3}-\big(\GQ-\GI_3\big)\Ge_3$,  the estimate \eqref{WQ} leads to \eqref{EstWS}.
From the third estimate in \eqref{EstmRod} we obtain
 \begin{equation}\label{EstmWQ}
||\GQ-\GQ(0)||_{L^2(-\delta,L;\R^{3\times 3})}  \le {C\over \delta^2} \Gd(v, B_\delta),
\end{equation} Due to the definition \eqref{WBWS} of  ${\cal W}^{(m)}$ and  estimate $\eqref{EstmRod}_3$ we  get 
$$\Big\| {d{\cal W}^{(m)}\over dx_3}\Big\|_{H^1( -\delta,L)}\le {C\over \delta^2} \Gd(v, B_\delta)+C|||\GQ(0)-\GI_3|||$$ and thus \eqref{dWB}.
A straightforward calculation gives 
\begin{equation}\label{dW31}
{d{\cal W}^{(m)}_{3}\over dx_3}=\big(\GQ-\GI_3\big)\Ge_3\cdot\Ge_3=-{1\over 2}\big|(\GQ-\GI_3)\Ge_3\big|^2.
\end{equation} Besides we have
\begin{equation}\label{aboveequality}
{d\over dx_3}\big((\GQ-\GI_3)\Ge_3\big)={d\GQ\over dx_3}\Ge_3.
\end{equation} We recall that  for $\phi\in H^1(0,L)$ and  $\eta \in ]0,L[$ we have
\begin{equation}\label{aboveequality1}
\int_0^\eta|\phi(t)-\phi(0)|^2dt\le {\eta^2\over 2}\Big\|{d\phi\over dt}\Big\|^2_{L^2(0,L)},\qquad \int_0^\eta|\phi(t)-\phi(0)|^4dt\le {\eta^3\over 3}\Big\|{d\phi\over dt}\Big\|^4_{L^2(0,L)}.
\end{equation}
Then, the estimates $\eqref{EstmRod}_3$, $\eqref{aboveequality1}_1$ and the equality \eqref{aboveequality} give
 \begin{equation}\label{314}
 \begin{aligned}
 \big\|(\GQ-\GI_3)\Ge_3-(\GQ(0)-\GI_3)\Ge_3 \big\|_{L^2(-\delta,L;\R^3)}&\le {C\over \delta^2} \Gd(v, B_\delta),\\
  \big\|(\GQ-\GI_3)\Ge_3-(\GQ(0)-\GI_3)\Ge_3 \big\|_{L^2(-\delta,\delta;\R^3)}&\le {C\over \delta} \Gd(v, B_\delta).
\end{aligned}\end{equation} Now,  again $\eqref{EstmRod}_3$ and $\eqref{aboveequality1}_2$ lead to
 \begin{equation}\label{320}
 \begin{aligned}
 \big\|(\GQ-\GI_3)\Ge_3-(\GQ(0)-\GI_3)\Ge_3 \big\|_{L^4(-\delta,L;\R^3)}&\le {C \over \delta^2}\Gd(v, B_\delta),\\
  \big\|(\GQ-\GI_3)\Ge_3-(\GQ(0)-\GI_3)\Ge_3 \big\|_{L^4(-\delta,\delta;\R^3)}&\le {C \over \delta^{5/4}} \Gd(v, B_\delta).
\end{aligned} \end{equation}
Finally, from \eqref{dW31} and the above inequality we obtain  \eqref{dW3}.
\end{proof} 
\subsection{First estimates in the junction area.}
\begin{lemma}\label{lemme2} We have the following estimate on $\GQ(0)-\GI_3$:
\begin{equation}\label{Q(0)First}
\begin{aligned}
&\big|\big(\GQ(0)-\GI_3\big)\Ge_3\cdot\Ge_3\big|\le {C\over \delta^{3/2}}\big(\GG_s(u,\Omega_\delta)+\Gd(v,B_{\delta})\big),\\
&|||\GQ(0)-\GI_3|||\le  {C\over \delta^{7/4}}\GG_s(u,\Omega_\delta)+ {C\over \delta^{3/2}}\Gd(v,B_{\delta})
\end{aligned}\end{equation} and those about ${\cal W}(0)$ 
\begin{equation}\label{W1(0)}
|{\cal W}_\alpha(0)|\le {C\over\delta^{3/4}}\GG_s(u,\Omega_\delta)+{C\over \delta^{1/2}}\Gd(v,B_{\delta})
\end{equation}
and
\begin{equation}\label{W3(0)-V3(0,0)}
\begin{aligned}
&|{\cal W}_3(0)-\widetilde {\cal U}_3(0,0)| \le   {C\over \delta^{2}} \big[\Gd(v,B_{\delta})\big]^2+{C\over \delta^{1/2}}\big(\Gd(v,B_{\delta})+\GG_s(u,\Omega_\delta)\big)+{C\over\delta}\Gd(v,\Omega_\delta),\\
&|{\cal W}_3(0)| \le  {C\over \delta^{3/2}}\Gd(v,\Omega_{\delta})+{C\over \delta^{2}} \big[\Gd(v,B_{\delta})\big]^2+{C\over \delta^{1/2}}\big(\Gd(v,B_{\delta})+\GG_s(u,\Omega_\delta)\big).
\end{aligned}\end{equation}
The constants  are independent of $\delta$.
\end{lemma}
\begin{proof} 
\noindent{\it Step 1. We prove the estimate on $\GQ(0)-\GI_3$.} 
We consider the last inequalities in Theorems \ref{Theorem 3.3.} and  \ref{Theorem II.2.2.}. They give
\begin{equation}\label{Q-R}
\bigl\|\GQ-\GI_3-\GA_{\cal R} \big\|_{ L^2(C_{\delta}; \R^9)}\le C\big(\GG_s(u,\Omega_\delta)+\Gd(v,B_{\delta})\big).
\end{equation}
Now, from the third estimate in \eqref{EstmRod}, we get
\begin{equation*}\label{V=WQ(0)}
\big\|\GQ-\GQ(0)\big\|_{L^2(-\delta,\delta;\R^{3\times 3})} \le C \delta \Bigl\|{d\GQ\over dx_3}\Big\|_{L^2(-\delta, \delta ;\R^{3\times 3})}\le {C \over\delta} \Gd(v,B_{\delta}).
\end{equation*} Hence
\begin{equation}\label{Q(0)-R}
\bigl\|\GQ(0)-\GI_3-\GA_{\cal R}\big\|^2_{ L^2(D_\delta; \R^{3\times 3})}\le {C\over \delta}\Big(\big[\GG_s(u,\Omega_\delta)\big]^2+\big[\Gd(v,B_{\delta})\big]^2\Big).
\end{equation} We recall that the matrix $\GA_{\cal R}$ is antisymmetric, then \eqref{Q(0)-R} leads to the first estimate in \eqref{Q(0)First}. 
Due to  \eqref{Estm} we have
\begin{equation}\label{RDeps}
||{\cal R}||^4_{L^2(D_\delta;\R^3)}\le C\delta^{3} ||{\cal R}||^4_{L^8(D_\delta;\R^3)}\le C\delta^{3}||{\cal R}||^4_{H^1(\omega;\R^3)}\le {C\over \delta^{3}}\big[\GG_s(u,\Omega_\delta)\big]^4.
\end{equation} 
Then, using  the  above estimate   and \eqref{Q(0)-R} we deduce the second estimate in \eqref{Q(0)First}.
\vskip 1mm
\noindent {\it Step 2. We prove the estimate \eqref{W1(0)} on ${\cal W}_\alpha(0)$.} 
\vskip 1mm
\noindent The two decompositions of $u=v-I_d$ (\eqref{FDec} and \eqref{DecR}) give, for a.e. $x\in C_{\delta}$
\begin{equation}\label{DecC}
\begin{aligned}
&{\cal U} (x_1,x_2)+x_3{\cal R}(x_1,x_2)\land\Ge_3+\overline{u}(x)\\
=&{\cal W} (x_3)+(\GQ(x_3)-\GI_3)(x_1\Ge_1+x_2\Ge_2)+\overline{\overline{v}}(x).
\end{aligned}\end{equation} 
\noindent 
Taking the averages on the cylinder $C_{\delta}$ of the terms in this equality \eqref{DecC} give
\begin{equation}\label{V-WFirst}
{\cal M}_{D_{\delta}}\big({\cal U}\big)={1\over |D_{\delta}|}\int_{D_\delta}{\cal U}(x_1,x_2) dx_1dx_2={\cal M}_{I_{\delta}}\big({\cal W}\big)={1\over 2\delta}\int_{-\delta}^\delta{\cal W}(x_3) dx_3.
\end{equation}
Besides, proceeding as for ${\cal R}$ in \eqref{RDeps} and from \eqref{Estm} we have
\begin{equation*}
||{\cal U}_\alpha||_{L^2(D_\delta )} \le C\delta^{1/4}\GG_s(u,\Omega_\delta).
\end{equation*} From this estimate we get
\begin{equation}\label{Walpha}
|{\cal M}_{I_{\delta}}\big({\cal W}_\alpha\big)|=|{\cal M}_{D_{\delta}}\big({\cal U}_\alpha\big)|\le {C\over\delta^{3/4}} \GG_s(u,\Omega_\delta).
\end{equation} We set $\ds y_\alpha(x_3)={\cal W}_\alpha(x_3)-x_3(\GQ(0)-\GI_3)\Ge_3\cdot \Ge_\alpha$. The  estimates \eqref{WQ}  and \eqref{314} lead to
\begin{equation*}
\Bigl\|{dy_\alpha \over dx_3}\Big\|_{L^2(-\delta, \delta)}\le {C\over \delta}\Gd(v,B_{\delta})
\end{equation*}  which in turn implies
\begin{equation*}
\bigl\|y_\alpha-y_\alpha(0)\big\|_{L^2(-\delta,\delta)}\le {C}\Gd(v,B_{\delta}).
\end{equation*} Taking the average, it yields
\begin{equation}\label{yalpha-yalpha(0)}
|{\cal M}_{I_{\delta}}\big({\cal W}_\alpha\big)-{\cal W}_\alpha(0)|
\le  {C\over \delta^{1/2}}\Gd(v,B_{\delta}).
\end{equation}
Finally, from \eqref{Walpha} and \eqref{yalpha-yalpha(0)}  we obtain \eqref{W1(0)}.
\vskip 1mm
\noindent {\it Step 3. We prove the estimate on ${\cal W}_3(0)$.} Using \eqref{4002} we deduce that
\begin{equation}\label{DefU3}
\begin{aligned}
||{\cal U}_3-\widetilde{\cal U}_3||_{L^2(D_\delta)} & \le C\delta^{1/2}||{\cal U}_3-\widetilde{\cal U}_3||_{L^4(\omega)}\\
& \le C\delta^{1/2} ||{\cal U}_3-\widetilde{\cal U}_3||_{H^1(\omega)}\le C\Gd(v,\Omega_\delta).
\end{aligned}\end{equation} 
Then we replace ${\cal U}_3$ with $\widetilde{\cal U}_3$ and ${\cal W}_3$ with ${\cal W}^{(m)}_{3}$ in \eqref{V-WFirst}. Taking into account  \eqref{EstWS} we obtain
\begin{equation}\label{V-WSec}
|{\cal M}_{D_{\delta}}\big(\widetilde{\cal U}_3\big)-{\cal M}_{I_{\delta}}\big({\cal W}^{(m)}_{3}\big)|\le{C\over \delta}\Gd(v,\Omega_\delta)+{C\over \delta^{1/2}}\Gd(v, B_\delta).
\end{equation}
We carry on by comparing ${\cal M}_{D_{\delta}}\big(\widetilde{\cal U}_3\big)$ with $\widetilde{\cal U}_3(0,0)$.  Let us set
$${\Gr}_\alpha={1\over \pi\delta^2}\int_{D_\delta }\big(\GR(x_1,x_2)-\GI_3\big)\Ge_\alpha\cdot\Ge_3\,dx_1dx_2$$ and consider the function $\Psi(x_1,x_2)=\widetilde{\cal U}_3(x_1,x_2)-{\cal M}_{D_{\delta}}\big(\widetilde{\cal U}_3\big)-x_1\Gr_2-x_2\Gr_1$. Due to  \eqref{4002} we first obtain
\begin{equation}\label{Psi2}
\Big\|{\partial^2\Psi\over \partial x_\alpha\partial x_\beta}\Big\|_{L^2(D_\delta ,\R^3)}\le {C\over \delta^{3/2}}\Gd(v,\Omega_\delta).
\end{equation}
Then, applying twice the Poincar-Wirtinger inequality in the disc $D_\delta $ and using  \eqref{3.7} and the fourth estimate in \eqref{4002} lead to
\begin{equation}\label{Psi}
||\nabla\Psi||^2_{L^2(D_\delta ,\R^6)}\le {C\over \delta}\big[\Gd(v,\Omega_\delta)\big]^2,\qquad 
||\Psi||^2_{L^2(D_\delta ,\R^3)}\le C\delta\big[\Gd(v,\Omega_\delta)\big]^2.
\end{equation} From the above inequalities \eqref{Psi2} and \eqref{Psi} we deduce that
\begin{equation*}
\begin{aligned}
||\Psi||_{L^\infty(D_\delta ,\R^3)}\le {C \over \delta^{1/2}}\Gd(v,\Omega_\delta)\quad
\Longrightarrow\quad|\Psi(0,0)|= |\widetilde{\cal U}_3(0,0)-{\cal M}_{D_{\delta}}\big(\widetilde{\cal U}_3\big)|\le {C \over \delta^{1/2}}\Gd(v,\Omega_\delta).
\end{aligned}\end{equation*} From this last estimate and \eqref{V-WSec} we obtain
\begin{equation}\label{V(0)-WSec}
| \widetilde{\cal U}_3(0,0)-{\cal M}_{I_{\delta}}\big({\cal W}^{(m)}_{3}\big)|\le{C\over \delta}\Gd(v,\Omega_\delta)+ {C\over \delta^{1/2}}\Gd(v,B_{\delta}).
\end{equation}
Then using the second estimate in \eqref{dW3} and  \eqref{Q(0)First} we have
  \begin{equation}\label{dW33}
 \Bigl\|{d{\cal W}^{(m)}_{3}\over dx_3}\Big\|_{L^2(-\delta,\delta)}  \le  {C\over \delta^{5/2}} \big[\Gd(v,B_{\delta})\big]^2+{C\over \delta }\big(\GG_s(u,\Omega_\delta)+ \Gd(v,B_{\delta})\big).
 \end{equation} Finally, recalling that ${\cal W}_3(0)={\cal W}^{(m)}_{3}(0)$, the above inequality leads to
\begin{equation*}
|{\cal M}_{I_{\delta}}\big({\cal W}^{(m)}_{3}\big)-{\cal W}_3(0)|\le C\delta^{1/2}\Bigl\|{d{\cal W}^{(m)}_{3}\over dx_3}\Big\|_{L^2(-\delta,\delta)}\le  {C\over \delta^{2}} \big[\Gd(v,B_{\delta})\big]^2+{C\over \delta^{1/2} }\big(\GG_s(u,\Omega_\delta)+ \Gd(v,B_{\delta})\big)
\end{equation*}  which in turn with \eqref{V(0)-WSec} and \eqref{4002} lead to \eqref{W3(0)-V3(0,0)}.
\end{proof}
\subsection{Global estimates of $u$: Korn's type inequality.}
 
 \noindent Now, we give the last estimates of the displacement $u=v-I_d$ in the rod $B_{\delta}$.

\begin{lemma}\label{lemme41} For any deformation $v$ in $\D_{\delta}$ we have the following  inequalities for the displacement $u=v-I_d$ in the rod $B_{\delta}$:
\begin{equation}\label{KoP1}
\begin{aligned}
&||u-{\cal W}||_{L^2(B_\delta;\R^3)}\le C\big(\Gd(v, B_\delta)+  \delta^{1/4}\GG_s(u,\Omega_\delta)\big),\\
&\big\|{\cal W}_{\alpha} \big\|_{L^2(-\delta,L)}+\big\|{\cal W}^{(m)}_{\alpha} \big\|_{L^2(-\delta,L)} \le C\Big({\Gd(v, B_\delta)\over \delta^2} +{\GG_s(u,\Omega_\delta) \over \delta^{7/4}}\Big),\\
& \big\|{\cal W}_{3} \big\|_{L^2(-\delta ,L)}+ \big\|{\cal W}^{(m)}_{3} \big\|_{L^2(-\delta ,L)}\le C{ \big[\Gd(v, B_\delta)\big]^2\over \delta^4}+{C\over \delta^{3/2}}\big[\GG_s(u,\Omega_\delta)+\Gd(v,\Omega_{\delta})+\Gd(v,B_{\delta})\big].
\end{aligned}
\end{equation} The constants do not depend   on $\delta$.
\end{lemma}

\begin{proof} 
From \eqref{EstmWQ} and \eqref{Q(0)First} we get 
\begin{equation}\label{EstQ-I}
\big\|\GQ-\GI_3\big\|_{L^2(-\delta,L;\R^{3\times 3})}\le C\Big({ \Gd(v, B_\delta)
\over \delta^2}+ {\GG_s(u,\Omega_\delta)\over \delta^{7/4}}\Big).
\end{equation} Then, from \eqref{EstmRod} and the above inequality  we deduce that
\begin{equation}\label{u-W}
||u-{\cal W}||_{L^2(B_\delta;\R^3)}\le C\big(\Gd(v, B_\delta)+   \delta^{1/4}\GG_s(u,\Omega_\delta)\big).
\end{equation}
From \eqref{dWB} again  \eqref{Q(0)First} and \eqref{W1(0)} we obtain
\begin{equation}\label{dWBbis}
\big\|{\cal W}^{(m)}_{\alpha} \big\|_{H^1(-\delta,L)} \le {C\over \delta^2} \Gd(v, B_\delta)+{C\over \delta^{7/4}}\GG_s(u,\Omega_\delta).
 \end{equation}
 Then since ${\cal W}={\cal W}^{(m)}+{\cal W}^{(s)}$,  \eqref{EstWS}  and \eqref{dWBbis} give the second estimate in \eqref{KoP1}.
\vskip 1mm
\noindent  From \eqref{dW3} and \eqref{Q(0)First} we deduce that
\begin{equation*}
 \Big\|{d{\cal W}^{(m)}_{3}\over dx_3}\Big\|_{L^2(-\delta ,L)}\le  {C \over \delta^4} \big[\Gd(v, B_\delta)\big]^2+{C\over \delta^{3/2}}\big[\GG_s(u,\Omega_\delta)+\Gd(v, B_\delta)\big].\end{equation*}
which in turn using \eqref{W3(0)-V3(0,0)} lead to
\begin{equation*}
\begin{aligned}
 \big\|{\cal W}^{(m)}_{3} \big\|_{L^2(-\delta ,L)}\le {C \over \delta^4} \big[\Gd(v, B_\delta)\big]^2+{C\over \delta^{3/2}}\big[\GG_s(u,\Omega_\delta)+\Gd(v, \Omega_\delta)+\Gd(v, B_\delta)\big]
\end{aligned}\end{equation*} and then due to \eqref{EstWS} we get the last estimate in \eqref{KoP1}.
\end{proof} 

\begin{cor}\label{CorKorn} For any deformation $v$ in $\D_{\delta}$ we have the following Korn's type inequality for the displacement $u=v-I_d$ in the rod $B_{\delta}$:
\begin{equation}\label{KoP}
\begin{aligned}
&\big\|\nabla u\big\|_{L^2(B_\delta;\R^{3\times 3})}\le {C\over \delta} \Gd(v, B_\delta)
+ {C \over \delta^{3/4}}\GG_s(u,\Omega_\delta),\\
&||u_\alpha||_{L^2(B_\delta)}\le{C \over \delta} \Gd(v, B_\delta)+{C\over \delta^{3/4}}\GG_s(u,\Omega_\delta),\\
&||u_3||_{L^2(B_\delta)}\le C{ \big[\Gd(v, B_\delta)\big]^2\over \delta^3}+{C\over \delta^{1/2}}\big[\GG_s(u,\Omega_\delta)+\Gd(v,\Omega_{\delta})+\Gd(v,B_{\delta})\big].
\end{aligned}
\end{equation} The constants do not depend on $\delta$.
\end{cor}
\begin{proof}
From \eqref{EstmRod} and \eqref{EstQ-I}  we obtain
\begin{equation}\label{Gradu}
\big\|\nabla u\big\|_{L^2(B_\delta;\R^{3\times 3})}\le {C\over \delta} \Gd(v, B_\delta)
+ {C \over \delta^{3/4}}\GG_s(u,\Omega_\delta).
\end{equation} The second and third inequalities are   immediate consequences of Lemma \ref{lemme41}.
\end{proof}

\section{Elastic structure}
  \medskip
\subsection{Elastic energy.}
In this section  we assume that  the  structure ${\cal S}_{\delta}$ is made of an elastic material. The associated local energy $\widehat{W}\; :\;  \GX_3\longrightarrow \R^+$ is the following  St Venant-Kirchhoff's law \footnote{With a more general assumption on the nonlinear elasticity law (see for example \cite{FJM} page 1466)  we would obtain the same asymptotic behavior as in our case.} (see also \cite{C1}) 
\begin{equation}\label{HatW2}
\widehat{W}(F)=\left\{\begin{aligned}
&Q(F^TF-\GI_3)\quad\hbox{if}\quad \det(F)>0\\
&+\infty\hskip 1.8cm \hbox{if}\quad \det(F)\le 0.
\end{aligned}\right.
\end{equation} where $\GX_3$ is the space of  $3\times 3$ symmetric matrices and where the quadratic form $Q$ is given by
 \begin{equation}\label{QRQP}
Q(E)={\lambda \over 8}\big(tr(E) \big)^2+{\mu\over 4}tr\big(E^2\big)
, \end{equation}
and  where $(\lambda,\mu)$  are the Lam's coefficients of the material.
 Let us recall (see e.g. \cite{FJM} or \cite{BGRod}) that for any $3\times 3$ matrix $F$ such that $\det(F)>0$ we have
\begin{equation}\label{HatW4}
[tr(F^TF-\GI_3)]^2 = |||F^TF-\GI_3|||^2\ge \hbox { dist }(F,SO(3))^2.
\end{equation}

\subsection{Assumptions on the forces and final estimates.}

 Now we assume that the structure ${\cal S}_{\delta}$ is submitted to applied body forces $f_{\delta}\in L^2({\cal S}_{\delta} ; \R^3)$ and we define   the total energy  
$J_{\delta}(v)$\footnote{For later convenience, we have added the term $\ds \int_{{\cal S}_{\delta}}f_{\delta}(x)\cdot  I_d(x)dx$
to the usual standard energy, indeed this does not affect the  minimizing problem for  $J_{\delta}$. }  over $\D_\delta$ by
\begin{equation}\label{J}
J_{\delta}(v)=\int_{{\cal S}_{\delta}}\widehat{W}_\delta(\nabla  v)(x)dx-\int_{{\cal S}_{\delta}}f_{\delta}(x)\cdot (v(x)-I_d(x))dx.
\end{equation}

\noindent {\bf  Assumptions on the forces. } 
To introduce the scaling on $f_{\delta}$, let us consider   $f_r$, $g_1$, $g_2$  in $ L^2(0,L; \R^3)$ and  $f_p\in L^2(\omega;\R^3)$. We assume  that the force  $f_{\delta}$ is given  by
\begin{equation}\label{ForceP}
\begin{aligned}
f_{\delta}(x)&= \delta^{5/2} \Big[f_{r,1}(x_3)\Ge_1+f_{r,2}(x_3)\Ge_2+{1\over \delta^{1/2}}  f_{r,3}(x_3)\Ge_3+{x_1\over \delta^2} g_{1}(x_3)+{x_2\over \delta^2}g_{2}(x_3)\Big]\\
&\qquad x\in B_{\delta},\quad x_3>\delta,\\
f_{\delta,\alpha}(x)&=\delta^{2}  f_{p,\alpha}(x_1,x_2),\qquad f_{\delta,3}(x)=\delta^{3}  f_{p,3}(x_1,x_2),\qquad x\in \Omega_\delta.
\end{aligned}
\end{equation} We denote
\begin{equation}\label{F3}
F_{r,3}(x_3)=\int_{x_3}^Lf_{r,3}(s)ds,\qquad \hbox{for a. e. } x_3\in ]0,L[.
\end{equation} 
\begin{theorem}\label{TH53}
There exist two constants $C_0$ and $C_1$, which depend only on $\omega$, $L$ and $\mu$, such that if
\begin{equation}\label{C0}
||f_p||_{L^2(\Omega;\R^3)}\le C_0
\end{equation}
and if either 
\begin{equation}\label{C1}
\begin{aligned}
&\hbox{Case 1: for a. e. } x_3\in ]0,L[, \;\; F_{r,3}(x_3)\ge 0,\\
\hbox{or}&\\
 &\hbox{Case 2: } \;||f_{r,3}||_{L^2(0,L)}\le C_1
\end{aligned}\end{equation} then for $\delta$ small enough and  for any $v\in \D_\delta$ satisfying $J_\delta(v)\le 0$ we have
\begin{equation}\label{510}
\Gd(v, B_\delta) + \Gd(v,\Omega_\delta)  \le C\delta^{5/2}
\end{equation} where the constant does not depend on $\delta$.
\end{theorem}
\begin{proof} From \eqref{KoP0} and  the assumptions \eqref{ForceP}  on the body forces, we obtain on the one hand for any $v\in \D_\delta$ and with $u=v-I_d$
\begin{equation}\label{nivf}
\Big|\int_{\Omega_{\delta}}f_{\delta}(x)\cdot u(x)dx \Big|\le C\delta^{5/2}||f_p||_{L^2(\omega;\R^3)}\GG_s(u,\Omega_{\delta}).
\end{equation} 
As far as the term involving the forces in the rod are concerned we first have	
\begin{equation*}
\begin{aligned}
\int_{B_{\delta}}f_{\delta}(x)\cdot u(x) dx &=\pi\delta^{9/2}\int_{\delta}^L f_{r,\alpha}(x_3)  {\cal W}_\alpha(x_3)dx_3+\pi\delta^{4}\int_{\delta}^L f_{r,3}(x_3)  {\cal W}_3(x_3)dx_3\\
&+ \int_{B_{\delta}}f_{\delta}(x)\cdot (u(x)- {\cal W}(x_3))dx.
\end{aligned}\end{equation*} Then, using Lemma \ref{lemme41} and \eqref{ForceP} we first get
\begin{equation}\label{nivf2}
\begin{aligned}
&\Big|\int_{\delta}^Lf_{r,\alpha}(x_3) {\cal W}_\alpha(x_3)dx_3 \Big|\le {C\over \delta^2}\sum_{\alpha=1}^2||f_{r,\alpha}||_{L^2(0,L)}\big(\Gd(v, B_\delta)+ \delta^{1/4}\GG_s(u,\Omega_\delta)\big),\\
&\Big|\int_{B_{\delta}}f_{\delta}(x)\cdot \big(u(x)- {\cal W}(x_3)\big)dx \Big|\le C\delta^{5/2}\big( ||g_1||_{L^2(0,L;\R^3)}+||g_2||_{L^2(0,L;\R^3)}\big)\\
&\hskip6.2cm\big(\Gd(v, B_\delta)+  \delta^{1/4}\GG_s(u,\Omega_\delta)\big)
\end{aligned}\end{equation} Now we estimate $\ds\int_{\delta}^L f_{r,3}  {\cal W}_3(x_3)dx_3$. From \eqref{EstWS} we first obtain
\begin{equation}\label{F3WS}
\Big|\int_{\delta}^Lf_{r,3}(x_3) {\cal W}_3^{(s)}(x_3)dx_3 \Big|\le {C\over \delta} ||f_{r,3}||_{L^2(0,L)}\Gd(v, B_\delta).
\end{equation} Then we have
\begin{equation}\label{nivf2}
\int_{\delta}^Lf_{r,3}(x_3){\cal W}^{(m)}_{3}(x_3)dx_3= F_{r,3}(\delta){\cal W}^{(m)}_{3}(\delta)+\int_{\delta}^LF_{r,3}(x_3){d{\cal W}^{(m)}_{3}\over dx_3}(x_3)dx_3.
\end{equation} Taking to account \eqref{dW33}  and  \eqref{W3(0)-V3(0,0)}  we get 
\begin{equation}\label{WBdelta}
\begin{aligned}
|{\cal W}^{(m)}_{3}(\delta)| & \le |{\cal W}_3(0)|+\delta^{1/2}\Bigl\|{d{\cal W}^{(m)}_{3}\over dx_3}\Big\|_{L^2(-\delta,\delta)} \\
 & \le   {C\over \delta^{3/2}} \Gd(v,\Omega_{\delta})+ {C\over \delta^{2}} \big[\Gd(v,B_{\delta})\big]^2+{C\over \delta^{1/2} }\big(\GG_s(u,\Omega_\delta)+ \Gd(v,B_{\delta})\big).
 \end{aligned}\end{equation} 
Observe now that due to the expression \eqref{dW31} of $\ds {d{\cal W}^{(m)}_{3}\over dx_3}$, this derivative is nonpositive for a.e. $x_3\in ]0,L[$ (see \eqref{dW31}). 
\vskip 1mm
\noindent $\bullet$ If we are in Case 1 in \eqref{C1}, we have
\begin{equation*}
\int_{\delta}^Lf_{r,3}(x_3){\cal W}_{3}(x_3)dx_3\le C||f_{r,3}||_{L^2(0,L)}\Big[{ \Gd(v,\Omega_{\delta})\over \delta^{3/2}}+ {\big[\Gd(v,B_{\delta})\big]^2\over \delta^{2}} +{ \GG_s(u,\Omega_\delta)+ \Gd(v,B_{\delta})\over \delta^{1/2} }\Big].
\end{equation*} Hence, we obtain
\begin{equation*}
\begin{aligned}
\int_{B_{\delta}}f_{\delta}(x)\cdot u(x) dx \le& C\delta^{5/2}\sum_{\alpha=1}^2\big( ||f_{r,\alpha}||_{L^2(0,L)}+||g_\alpha||_{L^2(0,L;\R^3)}\big)\big(\Gd(v, B_\delta)+  \delta^{1/4}\GG_s(u,\Omega_\delta)\big) \\
+ C||f_{r,3} &||_{L^2(0,L)}\Big[\delta^{5/2} \Gd(v,\Omega_{\delta}) + \delta^2 [\Gd(v,B_{\delta})\big]^2 +\delta^{7/2}\big(\GG_s(u,\Omega_\delta)+ \Gd(v,B_{\delta})\big)\Big].
 \end{aligned}\end{equation*}  We recall that (see \cite{BGJE}) 
\begin{equation}\label{GGsd}
\GG_s(u,\Omega_\delta)\le C\Gd(v,\Omega_\delta)+{C\over \delta^{5/2}}\big[\Gd(v,\Omega_\delta)\big]^2
\end{equation} where the constant does not depend on $\delta$. Then due to \eqref{nivf} and the above inequalities  we obtain that
\begin{equation}\label{520}
\begin{aligned}
\int_{\Omega_{\delta}}f_{\delta}(x)\cdot u(x) dx \le & C||f_r||_{L^2(\omega;\R^3)}\delta^{5/2}\Gd(v, \Omega_\delta)+ C^*||f_r||_{L^2(\omega;\R^3)}\big[\Gd(v, \Omega_\delta)\big]^2\\
\int_{B_{\delta}}f_{\delta}(x)\cdot u(x) dx \le & C(f_r,g_1,g_2)\delta^{5/2}\big(\Gd(v, B_\delta)+  \Gd(v,\Omega_\delta)\big) \\    
&\hskip-1cm +C(f_r,g_1,g_2)\delta^{1/4}\big[\Gd(v,\Omega_\delta)\big]^2+  C||f_{r,3}||_{L^2(0,L)} \delta^2 \big[\Gd(v, B_\delta)\big]^2.
 \end{aligned}\end{equation} Now,  for any $v\in \D_\delta$ such that $J_\delta(v)\le 0$, assumptions \eqref{HatW2}, \eqref{QRQP},\eqref{HatW4} and the above estimates lead to
\begin{equation*}
\begin{aligned} 
&{\mu\over 8}\big(\big[\Gd(v, B_\delta)]^2+  \big[\Gd(v,\Omega_\delta)\big]^2\big)\le \int_{{\cal S}_\delta}\widehat{W}(\nabla v)(x)dx\le \int_{{\cal S}_\delta}f_\delta(x)\cdot u(x)dx\\
 \le & C(f_r,g_1,g_2)\delta^{5/2}\big(\Gd(v, B_\delta)+  \Gd(v,\Omega_\delta)\big) +C(f_r,g_1,g_2)\delta^{1/4}\big[\Gd(v,\Omega_\delta)\big]^2\\
+ & C||f_{r,3}||_{L^2(0,L)} \delta^2 \big[\Gd(v, B_\delta)\big]^2+C\delta^{5/2}||f_p||_{\omega;\R^3}\Gd(v,\Omega_\delta)+C^*||f_p||_{\omega;\R^3}\big[\Gd(v,\Omega_\delta)\big]^2.
 \end{aligned}\end{equation*} wich in turn gives
 \begin{equation*}
\begin{aligned} 
&\Big({\mu\over 8}-C||f_{r,3}||_{L^2(0,L)} \delta^2\Big)\big[\Gd(v, B_\delta)]^2+  \Big({\mu\over 8}-C^*||f_p||_{\omega;\R^3}-C(f_r,g_1,g_2)\delta^{1/4}\Big)\big[\Gd(v,\Omega_\delta)\big]^2\\
 \le & C(f_r,g_1,g_2)\delta^{5/2}\big(\Gd(v, B_\delta)+  \Gd(v,\Omega_\delta)\big)+C\delta^{5/2}||f_p||_{\omega;\R^3}\Gd(v,\Omega_\delta).
 \end{aligned}\end{equation*} Indeed the two quantities $C||f_{r,3}||_{L^2(0,L)} \delta^2$ and $C(f_r,g_1,g_2)\delta^{1/4}$ tend to $0$ as $\delta$ tends to $0$, then, under the condition $C^*||f_p||_{L^2(\omega;\R^3)}\le \mu/32$ and  for $\delta$ small enough we obtain
 \begin{equation*}
 \Gd(v, B_\delta) + \Gd(v,\Omega_\delta)  \le C\delta^{5/2}.
\end{equation*} The constant does not depend on $\delta$.
\vskip 1mm
\noindent $\bullet$ If we are in Case 2 in \eqref{C1}, from \eqref{KoP1} we immediately have 
\begin{equation*}
\int_{\delta}^Lf_{r,3}(x_3){\cal W}_{3}(x_3)dx_3\le C||f_{r,3}||_{L^2(0,L)}\Big[{\big[\Gd(v,B_{\delta})\big]^2\over \delta^{4}} +{ \GG_s(u,\Omega_\delta)+ \Gd(v,\Omega_\delta)+ \Gd(v,B_{\delta})\over \delta^{3/2} }\Big].
\end{equation*} Then, proceeding as in Case 1 leads to
\begin{equation}\label{521}
\begin{aligned}
\int_{B_{\delta}}f_{\delta}(x)\cdot u(x) dx \le& C\delta^{5/2}\sum_{\alpha=1}^2\big( ||f_{r,\alpha}||_{L^2(0,L)}+||g_\alpha||_{L^2(0,L;\R^3)}\big)\big(\Gd(v, B_\delta)+  \delta^{1/4}\GG_s(u,\Omega_\delta)\big) \\
&\hskip-1cm+ C||f_{r,3}||_{L^2(0,L)}\Big[\big[\Gd(v,B_{\delta})\big]^2  +\delta^{5/2}\big( \GG_s(u,\Omega_\delta)+ \Gd(v,\Omega_\delta)+ \Gd(v,B_{\delta})\big)\Big].
 \end{aligned}\end{equation} Then for $\delta$ small enough, we get
\begin{equation*}
\begin{aligned} 
& {\mu\over 8}\big(\big[\Gd(v, B_\delta)]^2+  \big[\Gd(v,\Omega_\delta)\big]^2\big) \le \int_{{\cal S}_\delta}\widehat{W}(\nabla v)(x)dx\le \int_{{\cal S}_\delta}f_\delta(x)\cdot u(x)dx\\
&\le  C\delta^{5/2}C(f,g)\big(\Gd(v, B_\delta)+  \Gd(v,\Omega_\delta)\big)+ C^{**}||f_{r,3}||_{L^2(0,L)}\Big[\big[\Gd(v,B_{\delta})\big]^2  +\big[ \Gd(v,\Omega_\delta)\big]^2\Big]\\
&+C\delta^{5/2}||f_p||_{\omega;\R^3}\Gd(v,\Omega_\delta)+C^*||f_p||_{\omega;\R^3}\big[\Gd(v,\Omega_\delta)\big]^2.
 \end{aligned}\end{equation*} Hence, under the conditions $C^*||f_p||_{L^2(\omega;\R^3)}\le \mu / 32$ and $C^{**}||f_{r,3}||_{L^2(0,L)}\le \mu/32$ we deduce that
$$\Gd(v, B_\delta) + \Gd(v,\Omega_\delta)  \le C\delta^{5/2}.$$ In the both cases, we finally obtain \eqref{510}
\end{proof} As a consequence of Theorem \ref{TH53} and estimates \eqref{520}-\eqref{521}, we deduce that for $\delta$ small enough and for any $v\in\D_\delta$ satisfying $J_\delta(v)\le 0$ we have $(u=v-I_d$)
\begin{equation}\label{EstDist}
 \int_{{\cal S}_{\delta}}f_{\delta}\cdot u\le C\delta^{5},\qquad 
\int_{{\cal S}_{\delta}}\widehat{W}_\delta(\nabla  v)(x)dx \le C\delta^{5}.\end{equation}
From \eqref{EstDist} we also obtain for any $v\in \D_\delta$ such that $J_\delta(v)\le 0$
\begin{equation}\label{EstJ}
c\delta^{5}\le J_{\delta}(v)
\end{equation} where $c$ is a nonpositive constant which does not depend on $\delta$. 
We set
$$m_{\delta}=\inf_{v\in \D_\delta}J_{\delta}(v).$$ 
As a consequence of  \eqref{EstJ} we  have 
\begin{equation}\label{mdelta}
c\le {m_{\delta}\over \delta^{5}}\le 0.
\end{equation}  In general, a minimizer of $J_{\delta}$ does not exist  on $\D_\delta$.

\section{Asymptotic behavior of a sequence of deformations of the whole structure ${\cal S}_{\delta}$.}

In this subsection and the following one, we consider a sequence of deformations $(v_\delta)$ belonging to $\D_\delta$ and satisfying 
\begin{equation}\label{energie}
\Gd(v_\delta, B_\delta) + \Gd(v_\delta,\Omega_\delta) \le C\delta^{5/2}
\end{equation} where the constant does not depend on $\delta$. Setting $u_\delta=v_\delta-I_d$, then, due to \eqref{energie} and \eqref{GGsd} we obtain that
\begin{equation}\label{energie2}
\GG_s(u_\delta, \Omega_\delta) \le C\delta^{5/2}.
\end{equation}
For any open subset ${\cal O}\subset\R^2$ and for any field $\psi\in H^1({\cal O};\R^3)$, we denote
\begin{equation}\label{NotGam}
\gamma_{\alpha\beta}(\psi)={1\over 2}\Big({\partial\psi_\alpha\over \partial x_\beta}+{\partial\psi_\beta\over \partial x_\alpha}\Big), \qquad (\alpha,\beta)\in\{1,2\}.
\end{equation}

\subsection{The rescaling operators}

Before rescaling the domains, we introduce the reference domain $\Omega$ for the plate and the one $B$ for the rod
$$\Omega=\omega\times ]-1,1[,\qquad B=D\times ]0, L[=D(O, 1)\times ]0, L[.$$
As usual when dealing with thin  structures, we rescale $\Omega_\delta$ and $B_{\delta}$ using -for the plate-  the operator
$$\Pi_\delta( w)(x_1,x_2,X_3)=w(x_1,x_2,\delta X_3)\hbox{ for any}\;\; (x_1, x_2, X_3)\in \Omega $$
defined for e.g. $w\in L^2(\Omega_\delta)$ for which $\Pi_\delta (w)\in  L^2(\Omega)$ and using -for the rod- the operator
$$P_\delta( w)(X_1,X_2,x_3)=w(\delta X_1, \delta X_2, x_3)\hbox{ for any}\;\; (X_1, X_2, x_3)\in B $$
defined for e.g. $w\in L^2(B_{\delta})$ for which $P_\delta (w)\in  L^2(B)$.

\subsection{Asymptotic behavior in the plate.}\label{In the plate.}

Following Section 2 we decompose the restriction of  $u_\delta=v_\delta-I_d$ to the plate.  The Theorem \ref{Theorem 3.3.} gives ${\cal U}_\delta$, ${\cal R}_\delta$ and $\overline{u}_\delta$, then estimates \eqref{Estm}  lead to the following convergences for  a subsequence still indexed by $\delta$ (see \cite{GSP} for the detailed proofs of the below convergences and equalities)
\begin{equation}\label{4.9}\begin{aligned}
{1\over \delta} {\cal U}_{\delta,3} &\longrightarrow   {\cal U}_3 \quad \hbox{strongly in}\quad H^1(\omega),\\
{1\over \delta^{2}}  {\cal U}_{\delta,\alpha} &\rightharpoonup    {\cal U}_\alpha \quad \hbox{weakly in}\quad H^1(\omega),\\
{1\over \delta} {\cal R}_{\delta} &\rightharpoonup  {\cal R} \quad \hbox{weakly in}\quad H^1(\omega;\R^3),\\
{1\over \delta^{3}}\Pi _\delta(\overline{u}_\delta)& \rightharpoonup  \overline{u} \quad \hbox{weakly in}\quad L^2(\omega;H^1(-1,1;\R^3),\\
{1\over \delta^{2}}\Bigl({\partial{\cal U}_\delta\over \partial x_\alpha}-{\cal R}_\delta\land \Ge_\alpha\Big)& \rightharpoonup {\cal Z}_\alpha\quad \hbox{weakly in}\quad L^2(\omega;\R^3).\end{aligned}\end{equation}
Denoting by $\GA_{\cal R}$ the field of antisymmetric matrices associated to ${\cal R}$ as in Section 2, we also have
\begin{equation}\label{4.10}\begin{aligned}
{1\over \delta^{2}}\Pi_\delta(u_\delta-{\cal U}_\delta)&\longrightarrow X_3{\cal R}\land\Ge_3\quad \hbox{strongly in}\quad L^2(\Omega;\R^3),\\
{1\over \delta}\Pi _\delta (\nabla  u_\delta)& \longrightarrow  \GA_{\cal R} \quad \hbox{strongly in}\quad L^2(\Omega;\R^9).\end{aligned}\end{equation}
The boundary conditions \eqref{CLUR} give here
\begin{equation}\label{CLURLimit}{\cal U}_3=0, \quad {\cal U}_\alpha=0,\quad  {\cal R}=0\qquad\hbox{on}\quad \gamma_0,
\end{equation}
while \eqref{4.9} show that ${\cal U}_3\in H^2(\omega)$ with 
\begin{equation}\label{4.100}
{\partial{\cal U}_3\over \partial x_1}=-{\cal R}_2, \qquad{\partial{\cal U}_3\over \partial x_2}={\cal R}_1.
\end{equation}

 In \cite{BGJE} (see Theorem 7.3) the  limit of the Green-St Venant's strain tensor of the sequence $v_\delta$ is also derived. Let us set
\begin{equation}\label{4.101}
\overline{u}_p=\overline{u}+{X_3\over 2}\big({\cal Z}_1 \cdot\Ge_3\big)\Ge_1+{X_3\over 2}\big({\cal Z}_2 \cdot\Ge_3\big)\Ge_2
\end{equation}  and 
\begin{equation}\label{401}
{\cal Z}_{\alpha\beta}=\gamma_{\alpha\beta}({\cal U})+{1\over 2}{\partial {\cal U}_3\over \partial x_\alpha}{\partial {\cal U}_3\over \partial x_\beta}.
\end{equation} 
Then  we have
$${1\over 2\delta^{2}}\Pi _\delta \big((\nabla v_\delta)^T\nabla v_\delta -\GI_3\big)\rightharpoonup
\GE_p\qquad\hbox{weakly in}\quad L^1(\Omega;\R^9),$$ where the symmetric matrix $\GE_p$ is defined by
\begin{equation}
\GE_p=\begin{pmatrix}
\displaystyle  -X_3{\partial^2{\cal U}_3\over \partial x_1^2}+{\cal Z}_{11} & \displaystyle  -X_3 {\partial^2 {\cal U}_3\over \partial x_1\partial x_2}+{\cal Z}_{12}
&\displaystyle  {1\over 2}{\partial\overline{u}_{p,1} \over \partial X_3}\\
* & \displaystyle  -X_3{\partial^2 {\cal U}_3\over \partial x_2^2}+{\cal Z}_{22}  &\displaystyle  {1\over 2}{\partial\overline{u}_{p,2} \over \partial X_3}\\
* & *&  \displaystyle  {\partial\overline{u}_{p,3} \over \partial X_3}
\end{pmatrix}.
\end{equation}

\subsection{Asymptotic behavior in the rod.}\label{In the rod.}

Now, we decompose the restriction of  $v_\delta=u_\delta+I_d$ to the rod (see Section 2).  The Theorem \ref{Theorem II.2.2.} gives ${\cal W}_\delta$, $\GQ_\delta$, $\overline{\overline{v}}_\delta$ and thanks to \eqref{WBWS} we define 
${\cal W}^{(m)}_{\delta}$ and ${\cal W}^{(s)}_{\delta}$. 
Then the estimates  in Theorem \ref{Theorem II.2.2.} and Lemma \ref{lem33} allow to claim that 
\begin{equation}\label{EstmRod2}
\begin{aligned}
&||\overline{\overline{v}}_\delta ||_{L^2(B_{\delta}; \R^3)}\le C \delta^{7/2},\quad||\nabla \overline{\overline{v}}_\delta ||_{ L^2(B_{\delta}; \R^3)}\le C \delta^{5/2},\quad  ||{\cal W}^{(s)}_{\delta}||_{H^1(-\delta,L;\R^3)}\le C \delta^{3/2},\\
& ||{\cal W}^{(m)}_{\delta,\alpha}-{\cal W}_{\delta,\alpha}(0)||_{H^2(-\delta,L)}\le  C\delta^{1/2}+C|||\GQ_\delta(0) -\GI_3|||,\\
&\Big\|{d{\cal W}^{(m)}_{\delta, 3}\over dx_3}\Big\|_{L^2(-\delta,L)}\le  C\delta+C |(\GQ_\delta(0) -\GI_3)\Ge_3\cdot\Ge_3|,\\
&||\GQ_\delta -\GQ_\delta (0)||_{H^1(-\delta,L;\R^9)}\le C\delta^{1/2}.
\end{aligned}
\end{equation}
 Moreover from Lemma \ref{lemme2} we get 
 \begin{equation}\label{QW1}\begin{aligned}
 &|||\GQ_\delta(0)-\GI_3|||\le C \delta^{3/4},\qquad |(\GQ_\delta(0)-\GI_3)\Ge_3\cdot\Ge_3|\le C \delta,\\
&|{\cal W}^{(m)}_{\delta,\alpha}(0)|\le  C\delta^{7/4},\qquad |{\cal W}_{\delta,3}(0)-\widetilde {\cal U}_{\delta, 3}(0,0)|\le C\delta^{3/2},\\
&|{\cal W}_{\delta,3}(0)|\le C\delta.
\end{aligned}\end{equation} Finally we obtain the following estimates of the terms ${\cal W}^{(m)}_{\delta,\alpha}$, ${\cal W}^{(m)}_{\delta,3}$ and $\GQ_\delta -\GI_3$:
\begin{equation}\label{EstmRod3}
\begin{aligned}
&||\GQ_\delta -\GI_3||_{H^1(-\delta,L;\R^9)}\le C\delta^{1/2},\quad ||{\cal W}^{(m)}_{\delta,\alpha}||_{H^2(-\delta,L)}\le  C\delta^{1/2},\\
&||{\cal W}^{(m)}_{\delta,3}||_{H^1(-\delta,L)}\le  C\delta.
\end{aligned}
\end{equation}

 \noindent Now we  are in a position to prove the  following lemma:
 
\begin{lemma}\label{Lemma 6.2. } There exists a subsequence still indexed by $\delta$ such that 
\begin{equation}\label{ConvPl}
\begin{aligned}
&{1\over \delta^{1/2}}{\cal W}^{(m)}_{\delta,\alpha}\rightharpoonup {\cal W}_{\alpha} \quad \hbox{weakly in}\quad  H^2(0,L),\\
&{1\over \delta^{1/2}}{\cal W}_{\delta,\alpha},\;{1\over \delta^{1/2}}{\cal W}^{(m)}_{\delta,\alpha}\longrightarrow   {\cal W}_{\alpha} \quad \hbox{strongly in}\quad  H^1(0,L),\\
&{1\over \delta}{\cal W}_{\delta,3},\;{1\over \delta}{\cal W}^{(m)}_{\delta,3}\rightharpoonup  {\cal W}_{3} \quad \hbox{weakly in}\quad  H^1(0,L),\\
&{1\over \delta^{3/2}}{\cal W}^{(s)}_{\delta}\rightharpoonup  {\cal W}^{(s)} \quad \hbox{weakly in}\quad  H^1(0,L;\R^3),\\
&{1\over \delta^{1/2}}\big(\GQ_{\delta}-\GI_3)\rightharpoonup \GA_{\cal Q} \quad \hbox{weakly in}\quad H^1(0,L ; \R^9), \\
&{1\over \delta^{5/2}}P_\delta(\overline{v}_{\delta}) \rightharpoonup  \overline{\overline{v}} \quad \hbox{weakly in}\quad  L^2(0,L;H^1(D;\R^3)).
\end{aligned}
\end{equation} We also have ${\cal W}_{\alpha}\in H^2(0, L)$ and for a.e. $x_3\in ]0,L[$ we have
\begin{equation}\label{VIs}
\begin{aligned}
{d{\cal W}_{1}\over dx_3}(x_3)={\cal Q}_2(x_3),\qquad {d{\cal W}_{2}\over dx_3}(x_3)=-{\cal Q}_1(x_3),\\
 {d{\cal W}_{3}\over dx_3}(x_3)+{1\over 2}\Big[\Big|{d{\cal W}_{1}\over dx_3}(x_3)\Big|^2+\Big|{d{\cal W}_{2}\over dx_3}(x_3)\Big|^2\Big]=0.
\end{aligned}
\end{equation}
The junction conditions
\begin{equation}\label{V=WQ(0)}
 {\cal W}_\alpha(0)=0,\quad {\cal Q}(0)=0,\qquad {\cal W}^{(s)}(0)=0,\qquad {\cal W}_{3}(0)={\cal U}_3(0,0)
\end{equation} hold true.  We have 
\begin{equation}\label{GSVcase2}
{1\over 2\delta^{3/2}}P_\delta \big((\nabla v_\delta)^T\nabla v_\delta-\GI_3\big)\rightharpoonup   \GE_r\qquad\hbox{weakly in}\quad  L^1(B ; \R^{3\times 3}),
\end{equation}  where the symmetric matrix $\GE_r$ is defined by  
\begin{equation}\label{HatEcase2}
\GE_r= \begin{pmatrix}
\ds   \gamma_{11}(\overline{\overline{v}}_r) & \ds   \gamma_{12}(\overline{\overline{v}}_r) & \ds -{1\over 2}X_2{d{\cal Q}_3\over dx_3}+{1\over 2}{\partial\overline{\overline{v}}_{r,3}\over \partial X_1}+{1\over 2}{d{\cal W}^{(s)}_1\over dx_3}\\  \\
* & \ds   \gamma_{22}(\overline{\overline{v}}_r) & \ds {1\over 2}X_1{d{\cal Q}_3\over dx_3}+{1\over 2}{\partial\overline{\overline{v}}_{r,3}\over \partial X_2}+{1\over 2}{d{\cal W}^{(s)}_2\over dx_3}\\  \\
* & * &  \ds -X_1{d^2{\cal W}_1\over dx^2_3}-X_2{d^2{\cal W}_2\over dx^2_3}+{d{\cal W}^{(s)}_3\over dx_3}&
\end{pmatrix}.
 \end{equation}
\end{lemma}
\begin{proof}   
First, taking into account \eqref{EstmRod2}, \eqref{EstmRod3} and upon extracting a subsequence it follows that the convergences  \eqref{ConvPl} hold true. First, due to the definition of ${\cal W}^{(m)}_\delta$ and the weak convergence in $H^1(0,L;\R^9)$ of the sequence $\ds {1\over \delta^{1/2}}\big(\GQ_{\delta}-\GI_3)$ towards the antisymmetric matrix  $\GA_{\cal Q}$ we deduce that $\ds {d{\cal W}\over dx_3}={\cal Q}\land \Ge_3$ wich gives the two first equalities in \eqref{VIs}. Then, the strong convergence in $L^\infty(0,L;\R^3)$ of the sequence $\ds {1\over \delta^{1/2}}\big(\GQ_{\delta}-\GI_3)\Ge_3$ towards ${\cal Q}\land\Ge_3$, hence $\ds {1\over \delta}|(\GQ_{\delta}-\GI_3)\Ge_3|^2$ convergences towards $|{\cal Q}\land\Ge_3|^2=\ds \Big|{d{\cal W}_1\over dx_3}\Big|^2+\Big|{d{\cal W}_2\over dx_3}\Big|^2$ strongly in $L^\infty(0,L)$. Finally  using equality \eqref{dW31} we obtain the last equality in  \eqref{VIs}.   The junction conditions on ${\cal Q}$ and  ${\cal W}_{\alpha}$ are immediate consequences of \eqref{QW1} and the convergences   \eqref{ConvPl}. 

In order to obtain the junction condition between the  bending in the plate and the stretching in  the rod note first that the  sequence $\ds{1\over \delta}\widetilde{\cal U}_{\delta,3}$  converges strongly in $H^1(\omega)$ to ${\cal U}_3$ because of \eqref{4002} and the first convergence in \eqref{4.9}. Besides this sequence is uniformly bounded in $H^2(D)$, hence it converges strongly to the same limit  ${\cal U}_3$ in $C^0(D)$.  
Moreover  the  weak convergence of the sequence $\ds{1\over \delta}{\cal W}^{(m)}_{\delta,3}$  in $H^1(0, L)$, implies the convergence of $\ds{1\over \delta}{\cal W}^{(m)}_{\delta,3}(0)={1\over \delta}{\cal W}_{\delta,3}(0)$ to ${\cal W}_{3}(0)$. Using the third estimate in \eqref{QW1} gives the last condition in \eqref{V=WQ(0)}.

 Once the convergences \eqref{ConvPl} are established, the limit of the rescaled Green-St Venant strain tensor of the sequence $v_\delta$ is analyzed  in \cite{BGRod} (see Subsection 3.3) and it gives \eqref{HatEcase2}.
\end{proof}

\section{Asymptotic behavior of the sequence $\ds{m_\delta\over\delta^{5}}$.}

The goal of this section is to establish  Theorem \ref{theo9.1}. Let us first introduce a few notations. We set 
\begin{equation}\label{deflim2}
\begin{aligned}
\D_0=\Big\{&({\cal U} ,{\cal W},{\cal Q}_3)\in H^1(\omega;\R^3)\times H^1(0,L ; \R^3) \times H^1(0,L)\;|\; \\
&{\cal U}_3\in H^2(\omega),\quad {\cal W}_\alpha\in H^2(0,L),\quad {\cal U}=0,\quad {\partial {\cal U}_3\over \partial x_\alpha}=0 \quad  \hbox{on} \enskip  \gamma_0,\\
&{d{\cal W}_{3}\over dx_3}+{1\over 2}\Big[\Big|{d{\cal W}_{1}\over dx_3}\Big|^2+\Big|{d{\cal W}_{2}\over dx_3}\Big|^2\Big]=0\quad \hbox{in}\enskip ]0,L[,\\
& {\cal W}_3(0)={\cal U}_3(0,0),\qquad  {\cal W}_\alpha(0)={d{\cal W}_\alpha\over dx_3}(0)={\cal Q}_3(0)=0\Big\}
\end{aligned}\end{equation} Let us notice that $\D_0$ is a closed subset of $H^1(\omega;\R^3)\times H^1(0,L ; \R^3) \times H^1(0,L)$.
\vskip 1mm
We introduce below the "limit" elastic energies for the plate and the rod whose expressions are well known for such structures\footnote{$E$, $\nu$ are the Young modulus and the Poisson's ratio of the plate and the rod.} 
 
\begin{equation}\label {JP}
\begin{aligned}
 {\cal J}_p({\cal U})&= {E\over 3(1-\nu^2)}\int_\omega\Big[(1-\nu)\sum_{\alpha,\beta=1}^2\Big|{\partial^2{\cal U}_3\over \partial x_\alpha\partial x_\beta}\Big|^2+\nu\big(\Delta{\cal U}_3\big)^2\Big]\\
&+{E\over (1-\nu^2)}\int_\omega\Big[(1-\nu)\sum_{\alpha,\beta=1}^2\big|{\cal Z}_{\alpha\beta}\big|^2+\nu\big({\cal Z}_{11}+{\cal Z}_{22}\big)^2\Big],\\
{\cal J}_r({\cal W}_1,{\cal W}_2,{\cal Q}_3) &= {E\pi\over 8}\int_0^L\Big[ \Big|{d^2{\cal W}_1 \over dx_3^2}\Big|^2+\Big|{d^2{\cal W}_2 \over dx_3^2}\Big|^2\Big]+{\mu\pi\over 8}\int_0^L\Big|{d{\cal Q}_3 \over dx_3}\Big|^2
\end{aligned}\end{equation}
where  ${\cal Z}_{\alpha\beta}$ is given by 
\begin{equation*}
{\cal Z}_{\alpha\beta}=\gamma_{\alpha\beta}({\cal U})+{1\over 2}{\partial {\cal U}_3\over \partial x_\alpha}{\partial {\cal U}_3\over \partial x_\beta}.
\end{equation*}
The total energy of the plate-rod structure is given by  the functional ${\cal J}$ defined over $\D_0$
\begin{equation}\label {J2}
 {\cal J}({\cal U},{\cal W},{\cal Q}_3)= {\cal J}_p({\cal U})+{\cal J}_r({\cal W}_1,{\cal W}_2,{\cal Q}_3) -{\cal L}({\cal U},{\cal W},{\cal Q}_3)
\end{equation} 
with
\begin{equation}\label{FormLin}
\begin{aligned}
 {\cal L}({\cal U},{\cal W},{\cal Q}_3)&=2\int_{\omega} f_p\cdot {\cal U} +\pi \int_0^L f_r\cdot {\cal W}dx_3+{\pi \over 2} \int_0^L{g_\alpha}\cdot\big({\cal Q}\land\Ge_\alpha\big)dx_3
 \end{aligned}
 \end{equation}  where
 $${\cal Q}=- {d{\cal W}_2\over dx_3}\Ge_1+ {d{\cal W}_1\over dx_3}\Ge_2+{\cal Q}_3\Ge_3.$$ 

Below we  prove the existence of at least a minimizer of  $ {\cal J}$. 
\begin{lemma}\label{LemExistence}  There exist two constants $C^{*}_p$, $C^{*}_r$  such that, if $(f_{p,1},f_{p,2})$ satisfies
\begin{equation}\label{CstPl}
||f_{p,1}||^2_{L^2(\omega)}+||f_{p,2}||^2_{L^2(\omega)} <  C^*_p
\end{equation} and if $f_{r,3}$ satisfies
\begin{equation}\label{CstPt}
||f_{r,3}||_{L^2(0,L)} <  C^{*}_r
\end{equation} then the minimization problem
\begin{equation}\label{L71}
\min_{({\cal U},{\cal W},{\cal Q}_3)\in \D_0} {\cal J}({\cal U},{\cal W},{\cal Q}_3)
\end{equation} admits at least a solution.
\end{lemma}
\begin{proof}  Due to the boundary conditions on ${\cal U}_3$ in $\D_0$, we immediately have
\begin{equation}\label{NU3}
||{\cal U}_3||^2_{H^2(\omega)}\le C{\cal J}_p({\cal U}).
\end{equation} Then we get
\begin{equation}\label{NU4}
\begin{aligned}
\sum_{\alpha,\beta=1}^2||\gamma_{\alpha,\beta}({\cal U})||^2_{L^2(\omega)} & \le C{\cal J}_p({\cal U})+ C\big\|\nabla{\cal U}_3\big\|^4_{L^4(\omega;\R^2)}\\
& \le C{\cal J}_p({\cal U})+ C[{\cal J}_p({\cal U})]^2.
\end{aligned}\end{equation} Thanks to the 2D Korn's inequality we obtain
\begin{equation}\label{NU5}
||{\cal U}_1||^2_{H^1(\omega)}+||{\cal U}_2||^2_{H^1(\omega)}  \le C{\cal J}_p({\cal U})+ C_{p}[{\cal J}_p({\cal U})]^2.
\end{equation} 

\noindent  Again, due to the boundary conditions on ${\cal W}_\alpha$ and ${\cal Q}_3$ in $\D_0$, we immediately have
\begin{equation}\label{NU6}
||{\cal W}_1||^2_{H^2(0,L)}+||{\cal W}_2||^2_{H^2(0,L)}+||{\cal Q}_3||^2_{H^1(0,L)}\le C{\cal J}_r({\cal W}_1,{\cal W}_2,{\cal Q}_3).
\end{equation} Then, due to the definition of $\D_0$ and \eqref{NU6} we get
\begin{equation}\label{NU7}
\begin{aligned}
\Big\|{d{\cal W}_3\over dx_3}\Big\|^2_{L^2(0,L)} & \le  C\Bigl\{\Big\|{d{\cal W}_1\over dx_3}\Big\|^4_{L^4(0,L)}+\Big\|{d{\cal W}_2\over dx_3}\Big\|^4_{L^4(0,L)}\Big\} \le  C[{\cal J}_r({\cal W}_1,{\cal W}_2,{\cal Q}_3)]^2.
\end{aligned}\end{equation} From the above inequality and \eqref{NU3} we obtain
\begin{equation}\label{NU8}
\begin{aligned}
\big\|{\cal W}_3\big\|^2_{L^2(0,L)} &\le C|{\cal W}_3(0)|^2+C\Big\|{d{\cal W}_3\over dx_3}\Big\|^2_{L^2(0,L)}\\
& \le C{\cal J}_p({\cal U})+C_r[{\cal J}_r({\cal W}_1,{\cal W}_2,{\cal Q}_3)]^2.
\end{aligned}\end{equation}

Since $ {\cal J}(0, 0, 0)=0$, let us consider a minimizing sequence $({\cal U}^{(N)},{\cal W}^{(N)},{\cal Q}^{(N)}_3)\in \D_0$ satisfying ${\cal J}({\cal U}^{(N)},{\cal W}^{(N)},{\cal Q}^{(N)}_3)\le 0$ 
\begin{equation*}
m=\inf_{({\cal U},{\cal W},{\cal Q}_3)\in\D_0}{\cal J}({\cal U},{\cal W},{\cal Q}_3)=\lim_{N\to+\infty}{\cal J}({\cal U}^{(N)},{\cal W}^{(N)},{\cal Q}^{(N)}_3)
\end{equation*} where $m\in [-\infty, 0]$.

\noindent With the help of \eqref{NU3}-\eqref{NU8} we get
\begin{equation}
\begin{aligned}
&\hskip-5mm{\cal J}_p( {\cal U}^{(N)})+ {\cal J}_r({\cal W}^{(N)}_1{\cal W}^{(N)}_2, {\cal Q}^{(N)}_3)\le  C ||f_{p,3}||\sqrt{{\cal J}_p({\cal U}^{(N)})}\\
&+2\big(||f_{p,1}||^2_{L^2(\omega)}+||f_{p,2}||^2_{L^2(\omega)}\big)^{1/2}\big(C\sqrt{{\cal J}_p({\cal U}^{(N)})}+\sqrt{C_p}{\cal J}_p({\cal U}^{(N)})\big)\\
&+C\sum_{\alpha=1}^2\big(||f_{r,\alpha}||_{L^2(0,L)}+||g_{\alpha}||_{L^2(0,L;\R^3)}\big)\sqrt{{\cal J}_r({\cal W}^{(N)}_1{\cal W}^{(N)}_2,{\cal Q}^{(N)}_3)}\\
&+\pi||f_{r,3}||_{L^2(0,L)}\big(C\sqrt{{\cal J}_r({\cal W}^{(N)}_1{\cal W}^{(N)}_2,{\cal Q}^{(N)}_3)}+\sqrt{C_r}{\cal J}_r({\cal W}^{(N)}_1{\cal W}^{(N)}_2,{\cal Q}^{(N)}_3)\big)
\end{aligned}\end{equation} Choosing $\ds C^*_p={1\over 2C_p}$ and $\ds C^{*}_r={1\over \pi\sqrt{C_r}}$, if the applied forces satisfy \eqref{CstPl} and \eqref{CstPt}  then the following estimates hold true
\begin{equation}
\begin{aligned}
&||{\cal U}^{(N)}_3||_{H^2(\omega)}+||{\cal U}^{(N)}_1||_{H^1(\omega)}+||{\cal U}^{(N)}_2||_{H^1(\omega)} +||{\cal W}^{(N)}_1||_{H^2(0,L)}\\
+ &||{\cal W}^{(N)}_2||_{H^2(0,L)}+||{\cal Q}^{(N)}_3||_{H^1(0,L)}+||{\cal W}^{(N)}_3||_{H^1(0,L)}\le C
\end{aligned}\end{equation} where the constant $C$ does not depend on $N$.
\vskip 1mm
As a consequence, there exists $({\cal U}^{(*)},{\cal W}^{(*)},{\cal Q}^{(*)}_3)\in H^1(\omega;\R^3)\times H^1(0,L;\R^3)\times H^1(0,L)$ such that for a subsequence 
\begin{equation*}
\begin{aligned}
{\cal U}^{(N)}_3 &\rightharpoonup {\cal U}^{(*)}_3\quad \hbox{weakly in}\enskip H^2(\omega) \;\hbox{and strongly in }\; W^{1,4}(\omega),\\
{\cal U}^{(N)}_\alpha &\rightharpoonup {\cal U}^{(*)}_\alpha\quad \hbox{weakly in}\enskip H^1(\omega),\\
{\cal W}^{(N)}_\alpha&\rightharpoonup {\cal W}^{(*)}_\alpha\quad \hbox{weakly in}\enskip H^2(0,L) \;\hbox{and strongly in }\; W^{1,4}(0,L),\\
{\cal Q}^{(N)}_3&\rightharpoonup {\cal Q}^{(*)}_3\quad \hbox{weakly in}\enskip H^1(0,L),\\
{\cal W}^{(N)}_3&\rightharpoonup {\cal W}^{(*)}_3\quad \hbox{weakly in}\enskip H^1(0,L).
\end{aligned}\end{equation*} Notice that we also get the following convergences:
$${\cal Z}^{(N)}_{\alpha\beta}\rightharpoonup {\cal Z}^{(*)}_{\alpha\beta}=\gamma_{\alpha\beta}({\cal U}^{(*)})+{1\over 2}{\partial {\cal U}^{(*)}_3\over \partial x_\alpha}{\partial {\cal U}^{(*)}_3\over \partial x_\beta} \qquad \hbox{weakly in}\enskip L^2(\omega).$$
The above convergences show that $({\cal U}^{(*)},{\cal W}^{(*)},{\cal Q}^{(*)}_3)\in \D_0$.
Finally, since ${\cal J}$ is weakly sequentially continuous in 
$$ H^1(\omega;\R^2) \times H^2(\omega)\times L^2(\omega;\R^3)\times H^2(0,L;\R^2)\times H^1(0,L;\R^2)$$ with respect to 
$$( {\cal U}_1, {\cal U}_2,{\cal U}_3, {\cal Z}_{11}, {\cal Z}_{12}, {\cal Z}_{22}, {\cal W}_1, {\cal W}_2, {\cal W}_3,{\cal Q}_3)$$ The above weak and strong converges imply that
\begin{equation*}
{\cal J}({\cal U}^{(*)},{\cal W}^{(*)},{\cal Q}^{(*)}_3)=m
=\min_{({\cal U},{\cal W},{\cal Q}_3)\in\D_0}{\cal J}({\cal U},{\cal W},{\cal Q}_3)
\end{equation*} which ends the proof of the lemma.
\end{proof}

\begin {theorem}\label{theo9.1}
We have 
\begin{equation}\label{res2}
\lim_{\delta\to 0}{m_{\delta}\over \delta^{5}}=\min_{({\cal U},{\cal W},{\cal Q}_3)\in \D_0 } {\cal J}({\cal U},{\cal W},{\cal Q}_3),
\end{equation}
where the functional ${\cal J}$ is defined by \eqref{J2}.
\end{theorem}
\begin{proof}
{\it Step 1.} In this step we show that 
\begin{equation}\label{step1}
\min_{({\cal U},{\cal W},{\cal Q}_3)\in \D_0 }  {\cal J}({\cal U},{\cal W},{\cal Q}_3)\le \liminf_{\delta\to0}{m_\delta\over \delta^{5}}.
\end{equation} 
Let   $(v_\delta)_{\delta>0}$  be  a sequence of deformations belonging to $\D_{\delta}$ and such that
\begin{equation}\label{Hypvdelta2}
\lim_{\delta\to 0}{J_{\delta}(v_\delta)\over \delta^{5}}=\liminf_{\delta\to0}{m_\delta\over \delta^{5}}.
\end{equation} One can always assume that $J_{\delta}(v_\delta)\le 0$ without loss of generality. From the analysis of the previous section and, in particular from  estimate \eqref{510} the sequence $v_\delta$ satisfies
 \begin{equation}\label{6.54}
\Gd( v_\delta , \Omega_{\delta})+\Gd( v_\delta ,  B_{\delta})\le C\delta^{5/2}.\end{equation} 
From \eqref{HatW2}-\eqref{QRQP} and estimates \eqref{EstDist},  we obtain
 \begin{equation}\label{6.55}
 \big\|\nabla v^T_\delta\nabla v_\delta-\GI_3\big\|_{L^2(\Omega_{\delta} ; \R^{3\times 3})}\le  C\delta^{5/2},\quad  \big\|\nabla v^T_\delta\nabla v_\delta-\GI_3\big\|_{L^2(B_{\delta} ; \R^{3\times 3})}\le  C\delta^{5/2}.
 \end{equation} 

 Firstly, for any fixed $\delta$, the displacement  $u_\delta=v_\delta-I_d$, restricted to $\Omega_\delta$, is decomposed as in Theorem \ref{Theorem 3.3.}.  Due to  estimate  \eqref{6.54}, we can apply the results of Subsection \ref{In the plate.} to the sequence $(v_\delta)$. As a consequence  there exist a subsequence (still indexed by $\delta$) and ${\cal U}^{(0)},\;{\cal R}^{(0)}\in H^1(\omega;\R^3)$ and $\overline{u}^{(0)}_p\in L^2(\omega;H^1(-1,1;\R^3))$,  such that the convergences \eqref{4.9} and \eqref{4.10} hold true. Due to \eqref{CLURLimit} and \eqref{4.100} the field ${\cal U}^{(0)}_3$ belongs to $H^2(\omega)$, and we have the boundary conditions 
\begin{equation}\label{CLU0}{\cal U}^{(0)}=0,\quad \nabla U^{(0)}_3=0,\qquad\hbox{on}\quad \gamma_0.
\end{equation} Subsection  \ref{In the plate.} and estimates in \eqref{6.55} also  show that 
 \begin{equation}\label{E0}
 {1\over 2\delta^{2}}\Pi_\delta\big(\nabla v_\delta^T\nabla v_\delta-\GI_3\big)\rightharpoonup \GE^{(0)}_p\quad \hbox{weakly in }\enskip L^2(\Omega;\R^9)
 \end{equation}
 where $\GE^{(0)}_p$ is defined 
 \begin{equation}\label{E00}
\GE^{(0)}_p=\begin{pmatrix}
\displaystyle  -X_3{\partial^2{\cal U}^{(0)}_3\over \partial x_1^2}+{\cal Z}^{(0)}_{11} & \displaystyle  -X_3 {\partial^2 {\cal U}^{(0)}_3\over \partial x_1\partial x_2}+{\cal Z}^{(0)}_{12}
&\displaystyle  {1\over 2}{\partial\overline{u}^{(0)}_{p,1} \over \partial X_3}\\
* & \displaystyle  -X_3{\partial^2 {\cal U}^{(0)}_3\over \partial x_2^2}+{\cal Z}^{(0)}_{22}  &\displaystyle  {1\over 2}{\partial\overline{u}^{(0)}_{p,2} \over \partial X_3}\\
* & *&  \displaystyle  {\partial\overline{u}^{(0)}_{p,3} \over \partial X_3}
\end{pmatrix}\end{equation}
with  
\begin{equation}\label{Z0}
{\cal Z}^{(0)}_{\alpha\beta}=\gamma_{\alpha\beta}({\cal U}^{(0)})+{1\over 2}{\partial {\cal U}^{(0)}_3\over \partial x_\alpha}{\partial {\cal U}^{(0)}_3\over \partial x_\beta}.\end{equation} 
\vskip 1mm
Secondly, still for $\delta$  fixed,  the displacement  $u_\delta=v_\delta-I_d$, restricted to $B_{\delta}$, is decomposed as in Theorem \ref{Theorem II.2.2.} and \eqref{WBWS}.  Again due  to the   estimate in   \eqref{6.54}, we can apply the results of Subsection \ref{In the rod.} to the sequence $(v_\delta)$. As a consequence  there exist a subsequence (still indexed by $\delta$) and ${\cal W}^{(0)},\;{\cal W}^{(s,0)},\;{\cal Q}^{(0)}\in H^1(0,L;\R^3)$ and $\overline{\overline{v}}_r^{(0)}\in L^2(0,L;H^1(D;\R^3))$  such that the convergences \eqref{ConvPl} hold true. As a consequence of  \eqref{VIs} the components  ${\cal W}^{(0)}_\alpha$ belong to $H^2(0,L)$ and we have
$${d{\cal W}^{(0)}\over dx_3}={\cal Q}^{(0)}\land\Ge_3\quad \hbox{ and } \quad  {d{\cal W}^{(0)}_{3}\over dx_3}(x_3)+{1\over 2}\Big[\Big|{d{\cal W}^{(0)}_{1}\over dx_3}(x_3)\Big|^2+\Big|{d{\cal W}^{(0)}_{2}\over dx_3}(x_3)\Big|^2\Big]=0.$$
The junction conditions \eqref{V=WQ(0)}  give
\begin{equation}\label{CJUQ0}
{\cal Q}^{(0)}(0)=0,\quad {\cal W}^{(0)}_\alpha(0)=0,\quad {\cal W}^{(s,0)}(0)=0,\qquad {\cal W}^{(0)}_3(0)={\cal U}^{(0)}_3(0,0).\end{equation}
\vskip 1mm
\noindent As a first consequence, the triplet $({\cal U}^{(0)}, {\cal W}^{(0)},{\cal Q}^{(0)}_3)$ belongs to $\D_0$.

 Subsection  \ref{In the rod.} and the second estimate \eqref{6.55} also  show that 
 \begin{equation}\label{612}
{1\over 2\delta^{3/2}}P_\delta \big((\nabla v_\delta)^T\nabla v_\delta-\GI_3\big)\rightharpoonup   \GE^{(0)}_r\qquad\hbox{weakly in}\quad  L^2(B ; \R^{3\times 3}),
\end{equation}  where the symmetric matrix $\GE^{(0)}_r$ is defined by  
\begin{equation}\label{613}
\GE^{(0)}_r= \begin{pmatrix}
\ds   \gamma_{11}(\overline{\overline{v}}_r^{(0)}) & \ds   \gamma_{12}(\overline{\overline{v}}_r^{(0)}) & \ds -{1\over 2}X_2{d{\cal Q}^{(0)}_3\over dx_3}+{1\over 2}{\partial\overline{\overline{v}}^{(0)}_{r,3}\over \partial X_1}+{1\over 2}{d{\cal W}^{(s,0)}_1\over dx_3}\\  \\
* & \ds   \gamma_{22}(\overline{\overline{v}}_r^{(0)}) & \ds {1\over 2}X_1{d{\cal Q}^{(0)}_3\over dx_3}+{1\over 2}{\partial\overline{\overline{v}}^{(0)}_{r,3}\over \partial X_2}+{1\over 2}{d{\cal W}^{(s,0)}_2\over dx_3}\\  \\
* & * &  \ds -X_1{d^2{\cal W}^{(0)}_1\over dx^2_3}-X_2{d^2{\cal W}^{(0)}_2\over dx^2_3}+{d{\cal W}^{(s,0)}_3\over dx_3} &
\end{pmatrix}.
\end{equation} 
\noindent In order to bound from below the quantity $\ds \liminf_{\delta \to 0}{J_\delta(v_\delta)\over \delta^{5}}$, using the assumptions on the forces \eqref{ForceP} and the convergences \eqref{4.9} and \eqref{ConvPl} we first have
\begin{equation}\label{728}
 \lim_{\delta\to 0}{1\over \delta^{5}}\int_{{\cal S}_{\delta}} f_{\delta}\cdot(v_\delta -I_d)= {\cal L}({\cal U}^{(0)},{\cal W}^{(0)},{\cal Q}^{(0)}_3)
 \end{equation}   where
$ {\cal L}({\cal U},{\cal W},{\cal Q}_3)$ is given by \eqref{FormLin} for any triplet in $\D_0$. 

\noindent As far as the elastic energy is concerned, we write
\begin{equation*}
\begin{aligned}
{1\over \delta^{5}} \int_{{\cal S}_{\delta}}\widehat{W}_\delta\big(\nabla v_\delta\big)&= {1\over \delta^{5}} \int_{{\Omega}_{\delta}}\widehat{W}_\delta\big(\nabla v_\delta\big)+{1\over \delta^{5}} \int_{B_{\delta}\setminus C_{\delta}}\widehat{W}_\delta\big(\nabla v_\delta\big)\\
= \int_{\Omega}Q\Big(\Pi_\delta\Big[{1\over \delta^{2}}&\big((\nabla v_\delta)^T\nabla v_\delta-\GI_3\Big]\Big)+ \int_{B} Q\Big( \chi_{B\setminus D\times ]0,\delta[}P_\delta\Big[{1\over \delta^{3/2}}\big((\nabla v_\delta)^T\nabla v_\delta-\GI_3\Big]\Big)\\
\end{aligned}
\end{equation*} From the weak convergences  of the Green-St Venant's tensors in \eqref{E0}, \eqref{612} and equality \eqref{728}, we obtain 
\begin{equation}\label{6.100}
\liminf_{\delta \to 0}{J_\delta(v_\delta)\over \delta^{5}}\ge\int_{\Omega}Q\big(\GE^{(0)}_p\big)+ \int_{B} Q\big(\GE^{(0)}_r\big)- {\cal L}({\cal U}^{(0)},{\cal W}^{(0)},{\cal Q}^{(0)}_3)
\end{equation} where $\GE^{(0)}_p$ and $\GE^{(0)}_r$ are given by \eqref{E00} and \eqref{613}.
\vskip 1mm
The next step in the derivation of the limit energy consists in minimizing $\ds\int_{-1}^1Q\big(\GE^{(0)}_p\big)dX_3$ (resp. $\ds\int_D Q\big(\GE^{(0)}_r\big)dX_1dX_2$) with respect to $\overline{u}^{(0)}_p$( resp. $\overline{\overline{v}}_r^{(0)}$).

First the expressions of $Q$ and of $\GE^{(0)}_p$ under a few calculations show that 
\begin{equation}\label{605}
\begin{aligned}
\int_{-1}^1Q\big(\GE^{(0)}_p\big)dX_3\ge &{E\over 3(1-\nu^2)}\Big[(1-\nu)\sum_{\alpha,\beta=1}^2\Big|{\partial^2{\cal U}^{(0)}_3\over \partial x_\alpha\partial x_\beta}\Big|^2+\nu\big(\Delta{\cal U}^{(0)}_3\big)^2\Big]\\
&+{E\over (1-\nu^2)}\Big[(1-\nu)\sum_{\alpha,\beta=1}^2\big|{\cal Z}^{(0)}_{\alpha\beta}\big|^2+\nu\big({\cal Z}^{(0)}_{11}+{\cal Z}^{(0)}_{22}\big)^2\Big]
\end{aligned}\end{equation}
the expression in the right hand side of \eqref{605} is obtained through replacing $\overline{u}^{(0)}_p$ by 
\begin{equation}\label{barup}
\overline{\overline{u}}^{(0)}_p(\cdot ,\cdot ,X_3)={\nu\over 1-\nu}\Big[\Big({X_3^2\over 2}-{1\over 6}\Big)\Delta {\cal U}^{(0)}_3-X_3\big({\cal Z}^{(0)}_{11}+{\cal Z}^{(0)}_{22}\big)\Big]\Ge_3.
\end{equation}
Following \cite{BGRod} (see equation (4.56) and (4.57)), choosing  ${\cal W}^{(s,0)}_3=0$ and $\overline{\overline{v}}_r^{(0)}$ such that 
\begin{equation}\label{barur}\begin{aligned}
\overline{\overline{\overline{v}}}^{(0)}_{r,1}&=-\nu\Big[{X_2^2-X_1^2 \over 2}{d^2{\cal W}^{(0)}_1 \over dx_3^2}-X_1X_2{d^2{\cal W}^{(0)}_2 \over dx_3^2}\Big],\\
\overline{\overline{\overline{v}}}^{(0)}_{r,2}&=-\nu\Big[{X_1^2-X_2^2 \over 2}{d^2{\cal W}^{(0)}_2 \over dx_3^2}-X_1X_2{d^2{\cal W}^{(0)}_1 \over dx_3^2}\Big],\\
\overline{\overline{\overline{v}}}^{(0)}_{r,3}&=-X_1{d{\cal W}^{(s,0)}_1 \over dx_3}-X_2{d{\cal W}^{(s,0)}_2 \over dx_3},
\end{aligned}\end{equation} permit to obtain 
\begin{equation}\label{606}
\begin{aligned}
\int_D Q\big(\GE^{(0)}_r\big)dX_1dX_2\ge &{E\pi\over 8}\Big[ \Big|{d^2{\cal W}^{(0)}_1 \over dx_3^2}\Big|^2+\Big|{d^2{\cal W}^{(0)}_2 \over dx_3^2}\Big|^2\Big]+{\mu\pi\over 8}\Big|{d{\cal Q}^{(0)}_3 \over dx_3}\Big|^2.
\end{aligned}\end{equation}
 In view of \eqref{6.100}, \eqref{605} and \eqref{606}, the proof of\eqref{step1} is achieved.
\vskip 2mm
\noindent {\it Step  2.} In this step we show that 
$$ \limsup_{\delta\to0}{m_\delta\over \delta^{5}}\le \min_{({\cal U},{\cal W},{\cal Q}_3)\in \D_0}{\cal J}({\cal U},{\cal W},{\cal Q}_3).$$
Let $({\cal U}^{(1)},{\cal W}^{(1)},{\cal Q}^{(1)}_3)\in \D_0$ such that 
$$ \min_{({\cal U},{\cal W},{\cal Q}_3)\in \D_0}{\cal J}({\cal U},{\cal W},{\cal Q}_3)={\cal J}({\cal U}^{(1)},{\cal W}^{(1)},{\cal Q}^{(1)}_3).$$ 

\noindent We consider a sequence $\big({\cal U}^{(n)},{\cal W}^{(n)},{\cal Q}^{(n)}_3\big)_{n\ge 2}$ of elements belonging to $\D_0$ such that

$\bullet$  ${\cal U}^{(n)}_\alpha \in W^{2,\infty}(\omega)\cap H^1_{\gamma_0}(\omega)$ and 
\begin{equation}\label{C1}
\begin{aligned}
&\nabla{\cal U}^{(n)}_\alpha=0\quad \hbox{in}\quad D_{1/n},\quad( \alpha,\beta) \in\{1,2\}^2,\\
&{\cal U}^{(n)}_\alpha \longrightarrow {\cal U}^{(1)}_\alpha \hbox{ strongly in } H^1(\omega),
\end{aligned}
\end{equation}

$\bullet$ ${\cal U}^{(n)}_3 \in W^{3,\infty}(\omega)\cap H^2_{\gamma_0}(\omega)$ and 
\begin{equation}\label{C2}
\begin{aligned}
&{\partial^2{\cal U}^{(n)}_3\over \partial x_\alpha\partial x_\beta}=0\quad \hbox{in}\quad D_{1/n},\quad( \alpha,\beta) \in\{1,2\}^2,\\
&{\cal U}^{(n)}_3 \longrightarrow {\cal U}^{(1)}_3 \hbox{ strongly in } H^2(\omega),
\end{aligned}\end{equation}

$\bullet$  ${\cal W}^{(n)}_\alpha \in W^{3,\infty}(-1/n,L )$ with ${\cal W}^{(n)}_\alpha =0$ in $[-1/n,1/n]$  and
\begin{equation}\label{C3}
{\cal W}^{(n)}_\alpha \longrightarrow {\cal W}^{(1)}_\alpha \hbox{ strongly in } H^2(0,L),
\end{equation}

$\bullet$  ${\cal Q}^{(n)}_3 \in W^{1,\infty}(-1/n,L )$ with ${\cal Q}^{(n)}_3 =0$ in $[-1/n,1/n]$  and
\begin{equation}\label{C4}
{\cal Q}^{(n)}_3 \longrightarrow {\cal Q}^{(1)}_3 \hbox{ strongly in } H^1(0,L),
\end{equation}

We define  ${\cal W}^{(n)}_3 \in W^{2,\infty}(-1/n,L)$ by
$${\cal W}^{(n)}_3 ={\cal U}^{(n)}_3(0,0)\quad \hbox{and}\quad 
{d{\cal W}^{(n)}_{3}\over dx_3}+{1\over 2}\Big[\Big|{d{\cal W}^{(n)}_{1}\over dx_3}\Big|^2+\Big|{d{\cal W}^{(n)}_{2}\over dx_3}\Big|^2\Big]=0\qquad \hbox{in}\quad ]-{1/n},L[.$$ Obviously we have
\begin{equation}\label{C5}
{\cal W}^{(n)}_3 \longrightarrow {\cal W}^{(1)}_3 \hbox{ strongly in } H^1(0,L).
\end{equation}

In order to define an admissible deformation of the whole structure, we introduce both fields  $\overline{\overline{u}}^{(1)}_p\in L^2(\omega;H^1(-1,1;\R^3))$ obtained through replacing ${\cal U}^{(0)}$ by ${\cal U}^{(1)}$ in \eqref{Z0}-\eqref{barup}  and $\overline{\overline{\overline{v}}}^{(1)}_{r,\alpha}\in L^2(0,L;H^1(D))$ obtained through replacing ${\cal W}^{(0)}$ and ${\cal Q}^{(0)}_3$ by ${\cal U}^{(1)}$ and ${\cal Q}^{(0)}_3$ in  \eqref{barur} and taking $\overline{\overline{\overline{v}}}^{(1)}_{r,3}=0$.
\vskip 1mm
Then, we consider two sequences of warpings $\overline{u}^{(n)}_p, \overline{\overline{v}}^{(n)}_r$ such that
\vskip 1mm
$\bullet$  $\overline{u}^{(n)}_p \in W^{1,\infty}(\Omega;\R^3)$ with $\overline{u}^{(n)}_p=0$ on $\partial\omega\times ]-1,1[$, $\overline{u}^{(n)}_p=0$ in the cylinder $D(O, 1/n)\times ]-1,1[$  and
$$\overline{u}^{(n)}_p \longrightarrow  \overline{\overline{u}}^{(1)}_p \hbox{ strongly in } L^2(\omega ; H^1(-1,1;\R^3)),$$ 

$\bullet$  $\overline{\overline{v}}^{(n)}_r \in W^{1,\infty}(]-1/n,L[\times D;\R^3)$ with $\overline{\overline{v}}^{(n)}_r=0$  in the cylinder $D\times ]-1/n,1/n[$  and
$$\overline{\overline{v}}^{(n)}_r \longrightarrow \overline{\overline{\overline{v}}}^{(1)}_r \hbox{ strongly in } L^2(0,L ; H^1(D;\R^3)).$$

\vskip 2mm
For $n$ fixed, let us consider the sequence of deformations of the plate $\Omega_\delta$. We set
\begin{equation}\label{testnP}
\begin{aligned}
v^{(n)}_{\delta,1}(x)&=x_1+\delta^{2}\Big[{\cal U}^{(n)}_1(x_1,x_2)-{x_3\over \delta}{\partial {\cal U}^{(n)}_3\over \partial x_1}(x_1,x_2)+\delta\overline{u}^{(n)}_{p,1}(x_1,x_2,{x_3\over \delta}\big)\Big],\\
v^{(n)}_{\delta,2}(x)&=x_2+\delta^{2}\Big[{\cal U}^{(n)}_2(x_1,x_2)-{x_3\over \delta}{\partial {\cal U}^{(n)}_3\over \partial x_2}(x_1,x_2)+\delta\overline{u}^{(n)}_{p,2}(x_1,x_2,{x_3\over \delta}\big)\Big],\\
v^{(n)}_{\delta,3}(x)&=x_3+\delta\Big[{\cal U}^{(n)}_3(x_1,x_2)+\delta^2\overline{u}^{(n)}_{p,3}(x_1,x_2,{x_3\over \delta}\big)\Big].
\end{aligned}\end{equation}
 If $\delta$ is small enough (in order to have $\delta\le 1/n$) the expression of $v^{(n)}_\delta$ 
 in  the cylinder $C_{\delta}$ is given by
\begin{equation}\label{testnCyl}
\begin{aligned}
v^{(n)}_{\delta,1}(x)&=x_1+\delta^{2}\Big[{\cal U}^{(n)}_1(0,0)-{x_3\over \delta}{\partial {\cal U}^{(n)}_3\over \partial x_1}(0,0)\Big],\\
v^{(n)}_{\delta,2}(x)&=x_2+\delta^{2}\Big[{\cal U}^{(n)}_2(0,0)-{x_3\over \delta}{\partial {\cal U}^{(n)}_3\over \partial x_2}(0,0)\Big],\\
v^{(n)}_{\delta,3}(x)&=x_3+\delta \Big[{\cal U}^{(n)}_3(0,0)+ x_1{\partial {\cal U}^{(n)}_3\over \partial x_1}(0,0)+x_2{\partial {\cal U}^{(n)}_3\over \partial x_2}(0,0)\Big].
\end{aligned}
\end{equation} 
 We denote  
$${\cal Q}^{(n)}=-{d{\cal W}^{(n)}_2\over dx_3}\Ge_1+{d{\cal W}^{(n)}_1\over dx_3}\Ge_2+{\cal Q}^{(n)}_3\Ge_3,\qquad {\cal R}^{(n)}={\partial {\cal U}^{(n)}_3\over \partial x_2}(0,0)\Ge_1- {\partial {\cal U}^{(n)}_3\over \partial x_1}(0,0)\Ge_2.$$ The field ${\cal Q}^{(n)}$ belongs to $W^{1,\infty}(-1/n,L;\R^3)$. Let $\GR_\delta^{(n)}$ be the matrix field defined by
\begin{equation}\label{GR}
\GR_\delta^{(n)}(0)=\GI_3,\qquad {d\GR_\delta^{(n)}\over dx_3}=
\GA_{F^{(n)}_\delta}\GR_\delta^{(n)}\quad \hbox{in}\quad [-1/n,L]
\end{equation} where $F^{(n)}_\delta=\ds  \delta^{1/2}{d{\cal Q}^{(n)}\over dx_3}+\delta{\cal R}^{(n)}$ (see \eqref{1}) and let ${\cal W}^{(n)}_\delta$ be defined in $[-1/n,L]$ by
\begin{equation}\label{Wnd}
{\cal W}^{(n)}_\delta(x_3)=\int_0^{x_3}\big(\GR_\delta^{(n)}(t)-\GI_3\big)\Ge_3dt+\delta^2{\cal U}^{(n)}_1(0,0)\Ge_1+\delta^2{\cal U}^{(n)}_2(0,0)\Ge_2+\delta{\cal U}^{(n)}_3(0,0)\Ge_3.
\end{equation} 

We have $\GR_\delta^{(n)}\in W^{1,\infty}(-1/n , L ; SO(3))$, ${\cal W}^{(n)}_\delta\in W^{2,\infty}(-1/n , L ; \R^3)$  and the following strong convergences (as $\delta$ tends towards $0$)
\begin{equation}\label{740}
\begin{aligned}
\GR_\delta^{(n)}&\longrightarrow \GI_3 \hbox{ strongly in } W^{1,\infty}(-1/n,L; SO(3)),\\
{1\over \delta^{1/2}}{d\GR_\delta^{(n)}\over dx_3}&\longrightarrow \GA_{{d{\cal Q}^{(n)}\over dx_3}} \hbox{ strongly in } L^\infty(-1/n,L; \R^9),\\
{1\over \delta^{1/2}}\big(\GR_\delta^{(n)}-\GI_3\big)&\longrightarrow \GA_{{\cal Q}^{(n)}}\hbox{ strongly in } W^{1,\infty}(-1/n,L; \R^9),\\
{1\over \delta}\big(\GR_\delta^{(n)}-\GI_3\big)\Ge_3\cdot\Ge_3&\longrightarrow -{1\over 2}||\GA_{{\cal Q}^{(n)}}\Ge_3||_2^2 \hbox{ strongly in } L^{\infty}(-1/n,L),\\
{1\over \delta^{1/2}}{\cal W}^{(n)}_{\delta,\alpha}&\longrightarrow {\cal W}^{(n)}_\alpha \hbox{ strongly in } W^{1,\infty}(-1/n,L),\\
{1\over \delta}{\cal W}^{(n)}_{\delta,3}&\longrightarrow {\cal W}^{(n)}_3 \hbox{ strongly in } W^{1,\infty}(-1/n,L).
\end{aligned}\end{equation} By definition, ${\cal Q}^{(n)}$ is equal to 0 in $[-1/n,1/n]$, hence we have
$$\forall x_3\in [-1/n , 1/n],\qquad \GR_\delta^{(n)}(x_3)=\exp\big(\delta\GA_{{\cal R}^{(n)}}x_3\big).$$ Now, we consider the fields $\overline{\GR}_\delta^{(n)}\in W^{1,\infty}(-\delta , L ; SO(3))$ and $\overline{\cal W}^{(n)}_\delta\in W^{2,\infty}(-\delta , L ; \R^3)$  defined by
\begin{equation}\label{GR}
\begin{aligned}
\overline{\GR}_\delta^{(n)}(x_3)&=\exp\big(\delta\GA_{{\cal R}^{(n)}}x_3\big)\quad \hbox{in}\quad [-\delta,L],\\\
\overline{\cal W}^{(n)}_\delta(x_3)&=\int_0^{x_3}\big(\overline{\GR}_\delta^{(n)}(t)-\GI_3\big)\Ge_3dt+ {\cal W}^{(n)}_\delta(0)\quad \hbox{in}\quad [-\delta,L].
\end{aligned}\end{equation} 
We introduce a last displacement $\widetilde{v}^{(n)}_{\delta}$ belonging to $W^{1,\infty}( B_\delta ; \R^3)$ (for $\delta\le 1/n$) 
$$
\begin{aligned}
\widetilde{v}^{(n)}_{\delta}(x)=
&\begin{pmatrix}
\ds \delta^{2}\big({\cal U}^{(n)}_1(0,0)-{x_3\over \delta}{\partial {\cal U}^{(n)}_3\over \partial x_1}(0,0)\big)\\
\ds \delta^{2}\big({\cal U}^{(n)}_2(0,0)-{x_3\over \delta}{\partial {\cal U}^{(n)}_3\over \partial x_2}(0,0)\big)\\
\ds \delta\big( {\cal U}^{(n)}_3(0,0)+x_1{\partial {\cal U}^{(n)}_3\over \partial x_1}(0,0)+x_2{\partial {\cal U}^{(n)}_3\over \partial x_2}(0,0)\big)
\end{pmatrix}\\
&-\overline{\cal W}^{(n)}_\delta(x_3)-\big(\overline{\GR}^{(n)}_\delta(x_3)-\GI_3)(x_1\Ge_1+x_2\Ge_2).
\end{aligned}
$$ We have 
$$||\nabla \widetilde{v}^{(n)}_{\delta}||_{L^\infty(B_\delta ; \R^9)}\le C^{(n)}\delta^2.$$
Now, we are in a position to define the deformation $v^{(n)}_\delta$ in the rod $B_\delta$. We set
\begin{equation}\label{testnR}
v^{(n)}_{\delta}(x)=x+{\cal W}^{(n)}_\delta(x_3)+\big(\GR^{(n)}_\delta(x_3)-\GI_3)(x_1\Ge_1+x_2\Ge_2)+
\delta^{5/2}\overline{\overline{v}}^{(n)}_r\big({x_1\over \delta},{x_2\over \delta},x_3\big)+\widetilde{v}^{(n)}_{\delta}(x).
\end{equation} In the cylinder $C_\delta$, the above expression of $v^{(n)}_\delta$ matches the one given by \eqref{testnCyl} if $\delta$ is small enough ( $\delta\le 1/n$).
\vskip 2mm

By construction the deformation $v^{(n)}_\delta$ belongs to $\D_{\delta}$ and satisfies
$$||\nabla v^{(n)}_\delta-\GI_3||_{L^\infty({\cal S}_{\delta};\R^9)}\le C(n)\delta^{1/2}.$$ Hence, for a.e. $x\in {\cal S}_{\delta}$ we have $\det\big(\nabla v^{(n)}_\delta(x)\big)>0$. Then we have
\begin{equation}\label{630}
m_\delta\le J_\delta(v^{(n)}_\delta).
\end{equation}

The expression \eqref{testnP} of the displacement  $v^{(n)}_\delta-I_d$ in the plate is similar to the decomposition \eqref{FDec} given in Section 2. Hence the results of Subsection \ref{In the plate.} and the regularity of the terms ${\cal U}^{(n)}$ and $\overline{u}^{(n)}$ lead to 
\begin{equation}\label{631}
{1\over 2\delta^{2}}\Pi _\delta \big((\nabla v^{(n)}_\delta)^T\nabla  v^{(n)}_\delta -\GI_3\big)\longrightarrow
\GE^{(n)}_p\qquad\hbox{strongly in}\quad L^\infty(\Omega;\R^9),
\end{equation} where the symmetric matrix $\GE^{(n)}_p$ is defined by
\begin{equation*}
\GE^{(n)}_p=\begin{pmatrix}
\displaystyle  -X_3{\partial^2{\cal U}^{(n)}_3\over \partial x_1^2}+{\cal Z}^{(n)}_{11} & \displaystyle  -X_3 {\partial^2 {\cal U}^{(n)}_3\over \partial x_1\partial x_2}+{\cal Z}^{(n)}_{12}
&\displaystyle  {1\over 2}{\partial\overline{u}^{(n)}_{p,1} \over \partial X_3}\\
* & \displaystyle  -X_3{\partial^2 {\cal U}^{(n)}_3\over \partial x_2^2}+{\cal Z}^{(n)}_{22}  &\displaystyle  {1\over 2}{\partial\overline{u}^{(n)}_{p,2} \over \partial X_3}\\
* & *&  \displaystyle  {\partial\overline{u}^{(n)}_{p,3} \over \partial X_3}
\end{pmatrix}\end{equation*} Remark that here $\overline{u}^{(n)}_p=\overline{u}^{(n)}$ (see \eqref{4.101} and where  the ${\cal Z}^{(n)}_{\alpha\beta}$'s are obtained  through replacing  ${\cal U}$  by  ${\cal U}^{(n)}$ in \eqref{401}. 

\noindent Now, in the rod  $B_{\delta}$ we have the following strong convergence in $L^\infty( B ;\R^9)$ (as $\delta$ tends towards $0$):
\begin{equation}\label{760}
{1\over \delta^{3/2}} P_\delta\big(\nabla v^{(n)}_\delta-\GR^{(n)}_\delta\big)\longrightarrow  {d\GR^{(n)}\over dx_3}(X_1\Ge_1+X_2\Ge_2)+
\begin{pmatrix}
\ds   \gamma_{11}(\overline{\overline{v}}^{(n)}_r) & \ds   \gamma_{12}(\overline{\overline{v}}^{(n)}_r) & \ds {1\over 2}{\partial\overline{\overline{v}}^{(n)}_{r,3}\over \partial X_1}\\  \\
* & \ds   \gamma_{22}(\overline{\overline{v}}^{(n)}_r) & \ds {1\over 2}{\partial\overline{\overline{v}}^{(n)}_{r,3}\over \partial X_2}\\  \\
* & * &  0 
\end{pmatrix}.
\end{equation} Then we obtain
\begin{equation}\label{731}
{1\over 2\delta^{3/2}} P_\delta \big((\nabla v^{(n)}_\delta)^T\nabla  v^{(n)}_\delta -\GI_3\big)\longrightarrow
\GE^{(n)}_r\qquad\hbox{strongly in}\quad L^\infty(0 ;\R^9),
\end{equation} where the symmetric matrix $\GE^{(n)}_r$ is defined by
\begin{equation}\label{613}
\GE^{(n)}_r= \begin{pmatrix}
\ds   \gamma_{11}(\overline{\overline{v}}^{(n)}_r) & \ds   \gamma_{12}(\overline{\overline{v}}^{(n)}_r) & \ds -{1\over 2}X_2{d{\cal Q}^{(n)}_3\over dx_3}+{1\over 2}{\partial\overline{\overline{v}}^{(n)}_{r,3}\over \partial X_1}\\  \\
* & \ds   \gamma_{22}(\overline{\overline{v}}^{(n)}_r) & \ds {1\over 2}X_1{d{\cal Q}^{(n)}_3\over dx_3}+{1\over 2}{\partial\overline{\overline{v}}^{(n)}_{r,3}\over \partial X_2}\\  \\
* & * &  \ds -X_1{d^2{\cal W}^{(n)}_1\over dx^2_3}-X_2{d^2{\cal W}^{(n)}_2\over dx^2_3}
\end{pmatrix}.
\end{equation}  Before passing to the limit, notice that
$$
\begin{aligned}
\Big|{1\over \delta^5}\int_{{\cal S}_{\delta}}\widehat{W}_\delta(\nabla  v^{(n)}_\delta)(x)dx-\int_\Omega Q\Big(\Pi_\delta \big((\nabla v^{(n)}_\delta)^T\nabla  v^{(n)}_\delta -\GI_3\big)\Big)-\int_B Q\Big(P_\delta \big((\nabla v^{(n)}_\delta)^T\nabla  v^{(n)}_\delta -\GI_3\big)\Big)\Big|\\
\le {C\over \delta^5}\int_{C_\delta}||| \big((\nabla v^{(n)}_\delta)^T\nabla  v^{(n)}_\delta -\GI_3\big)|||^2, 
\end{aligned}$$ then from the expression \eqref{testnCyl} of $v^{(n)}_\delta$ in $C_\delta$ and the strong convergences \eqref{631}-\eqref{731} we get
$$\lim_{\delta\to 0}{1\over \delta^5}\int_{{\cal S}_{\delta}}\widehat{W}_\delta(\nabla  v^{(n)}_\delta)(x)dx=\int_\Omega Q(\GE^{(n)}_p)+\int_B Q(\GE^{(n)}_r).$$
\vskip 1mm
 From the expressions of $v^{(n)}_\delta$ in the plate and in the rod, from the convergences \eqref{740} and taking to account the expressions of the applied forces \eqref{ForceP}  we get
 \begin{equation*}
 \lim_{\delta\to 0}{1\over \delta^{5}}\int_{{\cal S}_{\delta}} f_{\delta}\cdot(v^{(n)}_\delta -I_d)= {\cal L}({\cal U}^{(n)},{\cal W}^{(n)},{\cal Q}^{(n)}_3).
 \end{equation*}  Then, from the above limits and \eqref{630} we finally get
 \begin{equation}\label{LimStep2}
 \limsup_{\delta\to0}{m_\delta\over \delta^{5}}\le\lim_{\delta\to 0}{J_\delta (v^{(n)}_\delta)\over \delta^5}= \int_{\Omega}Q\big(\GE^{(n)}_p\big)+ \int_{B} Q\big(\GE^{(n)}_r\big)-{\cal L}({\cal U}^{(n)},{\cal W}^{(n)},{\cal Q}^{(n)}_3).
 \end{equation} Now, $n$ goes to infinity, due to the definitions of the warpings $\overline{\overline{u}}^{(1)}_p$ and $\overline{\overline{\overline{v}}}^{(1)}_r$ and  the strong convergences \eqref{C1}-\eqref{C2}-\eqref{C3}-\eqref{C4}-\eqref{C5} that give
 $$ \limsup_{\delta\to0}{m_\delta\over \delta^{5}}\le \int_{\Omega}Q\big(\GE^{(1)}_p\big)+ \int_{B} Q\big(\GE^{(1)}_r\big)-{\cal L}({\cal U}^{(1)},{\cal W}^{(1)},{\cal Q}^{(1)}_3)
 = {\cal J}({\cal U}^{(1)},{\cal W}^{(1)},{\cal Q}^{(1)}_3).$$
This conclude the proof of the theorem.
\end{proof}
\begin{remark} Let us point out that Theorem \ref{theo9.1} shows that for any minimizing sequence $(v_\delta)_{\delta>0}$ as in Step 1, the convergence of the rescaled Green-St Venant's strain tensor in \eqref{E0} is a strong convergence in $L^2(\Omega;\R^{3\times 3})$ and the convergence \eqref{612} is a strong convergence in $L^2(B;\R^{3\times 3})$.
\end{remark}

\end{document}